\documentclass[a4paper]{amsart}

\usepackage{amssymb}
\usepackage{amsthm}  
\usepackage{amsmath} 
\usepackage{amscd} 
\usepackage{float}
\usepackage[all]{xypic}
\usepackage{url}
\usepackage{hyperref}
\usepackage{color}
 \usepackage{graphicx}
\hypersetup{linktocpage}

\newcommand{\om}{\omega}
\newcommand{\ka}{\kappa}
\newcommand{\si}{\sigma}

\newcommand{\ba}{\mathcal{G}} 
\newcommand{\A}{{\rm Aut}(\ba,\om)}

\newcommand{\tK}{\tilde K}

\newcommand{\tH}{\tilde H}
\newcommand{\tG}{\tilde G}
\newcommand{\tP}{\tilde P}
\newcommand{\tHn}{\tilde H_0}
\newcommand{\tfh}{\tilde{\mathfrak{h}}}
\newcommand{\ta}{\tilde{\alpha}}
\newcommand{\ti}{\tilde{\iota}}
\newcommand{\fg}{\mathfrak g}
\newcommand{\fp}{\mathfrak p}
\newcommand{\fd}{\mathfrak d}
\newcommand{\fk}{\mathfrak k}
\newcommand{\tfk}{\tilde{\mathfrak{k}}}

\newcommand{\fl}{\mathfrak l}

\newcommand{\fh}{\mathfrak h}
\newcommand{\fn}{\mathfrak n}
\newcommand{\tfn}{\tilde{\mathfrak{n}}}

\newcommand{\fq}{\mathfrak q}
\newcommand{\fm}{\mathfrak{m}(\fk)}

\newcommand{\Ad}{{\rm Ad}}

\newcommand{\Aut}{{\rm Aut}}
\newcommand{\id}{{\rm id}}
\newcommand{\conj}{{\rm conj}}

\newcommand{\U}{\Upsilon}

\newcommand{\lz}{[\![}
\newcommand{\pz}{]\!]}

\newcommand{\R}{\mathbb{R}}

\newcommand{\ad}{{\rm ad}}

\newcommand{\C}{\mathbb{C}}

\theoremstyle{plain}

\theoremstyle{plain}
\newtheorem{thm*}{Theorem}[section]
\newtheorem{prop*}[thm*]{Proposition}
\newtheorem{lem*}[thm*]{Lemma}
\newtheorem{cor*}[thm*]{Corollary}

\theoremstyle{definition}
\newtheorem{def*}{Definition}
\newtheorem{example}{Example}

\theoremstyle{remark}
\newtheorem{rem*}{Remark}

\begin{document}

\title{Geometric properties of homogeneous parabolic geometries with generalized symmetries}
\author{Jan Gregorovi\v c and Lenka Zalabov\' a}
\address{J.G. Department of Mathematics and Statistics, Faculty of Science, Masaryk
University, Kotl\' a\v rsk\' a 2, Brno, 611 37, Czech Republic; L.Z.
Institute of Mathematics and Biomathematics, Faculty of Science,
University of South Bohemia in \v Cesk\' e Bud\v ejovice, Brani\v sovsk\' a 1760, \v Cesk\' e Bud\v ejovice, 370 05, Czech Republic }
 \email{jan.gregorovic@seznam.cz, lzalabova@gmail.com}
 \thanks{First author supported by the Grant agency of the Czech Republic under the grant GBP201/12/G028. Second author supported by the grant P201/11/P202 of the Czech Science Foundation (GA\v CR)}
\maketitle
\begin{abstract} 
We investigate geometric properties of homogeneous parabolic geometries with generalized symmetries. We show that they can be reduced to a simpler geometric structures and interpret them explicitly. For specific types of parabolic geometries, we prove that the reductions correspond to known generalizations of symmetric spaces. In addition, we illustrate our results on an explicit example and provide a complete classification of possible non--trivial cases.
\end{abstract}

\tableofcontents

\section{Introduction}

The reader should be familiar with the theory of parabolic geometries, cf.  \cite {parabook}. We will always consider the parabolic geometry $(\ba \to M, \om)$ of type $(G,P)$ on the connected smooth manifold $M$ satisfying the following assumptions: The group $G$ is simple  (not necessarily connected) Lie group, the parabolic geometry $(\ba \to M, \om)$ is regular and normal, and its automorphism group $\A$ acts transitively on $M$, i.e.,  the parabolic geometry is homogeneous. 

Let us fix arbitrary $x_0\in M$ for the rest of the article and denote by $\A_{x_0}$ the stabilizer of $x_0$ in $\A$. Then any element of $\A_{u_0}$ can be identified with $g_0\exp(Z)$ for $g_0$ in $G_0$ and $\exp(Z)$ in the unipotent radical of $P$, where $G_0$ is Levi part of a chosen reductive Levi decomposition of $P$. Let us point out that reductive Levi decomposition of $P$ always exists, see \cite[Theorem 3.1.3]{parabook}, and we later fix one by choice of grading of $\fg$. Nevertheless, since all reductive Levi subgroups are conjugated by elements of $P$, the following definition does not depend on the choice of the Levi subgroup $G_0\subset P$.  

\begin{def*}
Let $s$ be an element of the center $Z(G_0)$ of the Levi subgroup $G_0\subset P$. We say that the automorphism $\phi\in \A_{x_0}$ is \emph{$s$--symmetry} at $x_0$ if there is $u_0\in \ba$ covering $x_0$ such that $\phi(u_0)=u_0s$. All $s$--symmetries at $x_0$ for all possible elements $s$ in $Z(G_0)$ together are called \emph{generalized symmetries} at $x_0$.
\end{def*}

We gave in \cite[Theorem 4.1.]{GZ2} a significant condition for the existence of generalized symmetries on homogeneous parabolic geometries. To formulate this condition here, we need to introduce important choices and notation:

We fix the restricted root system of $\fg$ in which $\fp$ is a standard parabolic subalgebra of $\fg$, and denote by $\alpha_i$ the positive simple restricted roots numbered according to the convention from \cite{Lie-alg-3,parabook}. We denote by $\Xi$ the set of simple restricted roots corresponding to $\fp$ and we use the notation $\fp_\Xi:=\fp$, because we will work with several different parabolic subalgebras of $\fg$ later and we will need to distinguish between them. We denote by $\fg_{\Xi,i}$ the corresponding $|k|$--grading of $\fg$ by $\Xi$--heights, i.e., $\fg_{\Xi,0}$ is Lie algebra of $G_0$, and we use the notation $\fg_{\Xi,-}$ and $\fp_{\Xi,+}$ for the negative and positive parts of the grading. 

\begin{itemize}
\item We denote by $\fg_{\gamma}$ the root space of the root $\gamma$, and we denote by $V_{\Xi,\gamma}$ the indecomposable $G_{0}$--submodule of $\fg$ containing the root space $\fg_{\gamma}$.
\end{itemize}
Let us point out that the representation $\Ad$ of $G_{0}$ on $\fg$ is completely reducible and both $\fg_{\Xi,-}$ and $\fp_{\Xi,+}$ decompose into the sums of the modules $V_{\Xi,\gamma}$, but there can be indecomposable $G_{0}$--submodules of $\fg_{\Xi,0}$ that are not of the form $V_{\Xi,\gamma}$ for some restricted root $\gamma$.
 
Since these decompositions and their description play a crucial role in the text, let us demonstrate them in several examples on a particular type of parabolic geometries -- the so--called generalized path geometries, see \cite[Section 4.4.3]{parabook}.
\begin{example}
Consider $\fg=\frak{sl}(n+1,\mathbb{R})$ for $n\geq 2$ and choose $\Xi=\{\alpha_1,\alpha_2\}$. Then $\fg_{\Xi,-}=V_{\Xi,-\alpha_1-\alpha_2}\oplus V_{\Xi,-\alpha_1}\oplus V_{\Xi,-\alpha_2}$ and $\fp_{\Xi,+}=V_{\Xi,\alpha_1+\alpha_2}\oplus V_{\Xi,\alpha_1}\oplus V_{\Xi,\alpha_2}$. Moreover, if $n>2$, then $\fg_{\Xi,0}=V_{\Xi,\alpha_3}\oplus \mathbb{R}=V_{\Xi,-\alpha_3}\oplus \mathbb{R}$, and if $n=2$, then $\fg_{\Xi,0}=\mathbb{R}\oplus \mathbb{R}$ is the Cartan subalgebra.
\end{example}

Moreover, the map $\Ad_s$ for $s\in Z(G_0)$ is a certain multiple of identity on each submodule $V_{\Xi,\gamma}$. Let us remark that we do not assume that $\Ad: Z(G_0)\to Gl(\fg)$ is injective contrary to the article \cite{GZ2}, in which the injectivity also poses no restriction and only makes the article \cite{GZ2} less technical. Indeed, since $\om=\Ad_{g_0}^{-1}\circ \om=(r^{g_0})^*\om$ holds for $g_0\in Ker(\Ad_{Z(G_0)})$ and thus $Ker(\Ad_{Z(G_0)})$ is subgroup of $\A$ consisting of automorphisms of $(\ba\to M,\om)$ covering identity on $M$, the results for $Ker(\Ad_{Z(G_0)})\backslash \A$ from \cite{GZ2} can be extended to our situation and we get the following result from \cite[Theorem 4.1.]{GZ2}. Let us point out that the assumption $\fg$ simple implies that the group of automorphisms covering identity on $M$, which we call trivial automorphisms, is countable and discrete.

\begin{prop*}\label{1.3}
There is $s$--symmetry $\phi\in \A$ at $x_0$ such that $\Ad_s|_{V_{\Xi,\alpha_i}}=j_i\cdot \id_{V_{\Xi,\alpha_i}}$ for each $\alpha_i\in \Xi$ if and only if there is automorphism $\phi'\in \A_{x_0}$ such that $T_{x_0}\phi'$ acts as $j_i^{-1}\cdot \id$ on the distinguished subspace of $T_{x_0}^{-1}M$ corresponding to each $\alpha_i\in \Xi$. Moreover, $T_{x_0}\phi|_{T_{x_0}^{-1}M}=T_{x_0}\phi'|_{T_{x_0}^{-1}M}$ holds for such $\phi$ and $\phi'$.
\end{prop*}

The possibilities for the eigenvalues $j_i$ depend on $Z(G_0)$. However there is an algebraic construction that allows $j_i$ to acquire all the admissible values, see \cite[Proposition 3.2]{GZ2} and Section \ref{tri}.
\

\begin{example}
If $\fg=\frak{sl}(n+1,\mathbb{R})$ for $n\geq 2$ and $\Xi=\{\alpha_1,\alpha_2\}$, then the generalized symmetries at $x_0$ have at most two eigenvalues $j_1^{-1}$ and $j_2^{-1}$ on $T_{x_0}^{-1}M$ corresponding to the decomposition $\fg_{\Xi,-1}=V_{\Xi,-\alpha_1}\oplus V_{\Xi,-\alpha_2}$. 

If we choose the model $PGL(n+1,\R)/P_{1,2}$ for the generalized path geometries, then all these eigenvaules are realized by elements of $Z(G_0)$. However, if we would have chosen $G=SL(n+1,\R)$, then we would be missing some eigenvalues for $n$ odd.
\end{example}

Let us remark that $T_{x_0}\phi\neq T_{x_0}\phi'$, in general, and there is a more general class of automorphisms that share the same action as generalized symmetries at $x_0$ on $T^{-1}_{x_0}M$. However, the Proposition \ref{1.3} says that it is enough to consider only generalized symmetries to study automorphisms from this class, because the existence of an automorphism from the more general class implies existence of a generalized symmetry, see \cite[Theorem 4.1]{GZ2}. Let us point out that it follows from the definition of the generalized symmetry that the point $u_0$ is the frame in which $T_{x_0}\phi$ coincides with $\Ad_s$.

On the other hand, there is the characterization \cite[Lemma 3.6.]{GZ2} of  automorphisms which act as $\id$ on $T_{x_0}^{-1}M$, but which generally are not trivial generalized symmetries, i.e., $s$--symmetries for $\Ad_s=\id$.

\begin{prop*}\label{1.4}
There is $\phi\in \A_{x_0}$ such that $\phi(u_0)=u_0\exp(Z)$ for some $Z\in \fp_{\Xi,+}$ and $u_0\in \ba$ covering $x_0$ if and only if $T_{x_0}\phi$ acts as $\id$ on $T_{x_0}^{-1}M$. Moreover, there is complete infinitesimal automorphism $\xi$ of $(\ba\to M,\om)$ such that $\om_{u_0}(\xi)=Z$.
\end{prop*}

In this article, we will show, how the information about the two distinguished classes of automorphisms can be used to deduce some additional invariant geometrical properties of $(\ba\to M,\om)$. We will obtain the geometrical properties by means of distinguished holonomy reductions (in the sense of \cite{Cap-holonomy}).

Let us summarize the notation and the information that we need to formulate the main result.

\begin{itemize}
\item  Let $\Phi$ be the set of all $\alpha_i\in \Xi$ such that there is a complete infinitesimal automorphism $\xi$ of $(\ba\to M,\om)$ with $\om_{u_0}(\xi)\in \fp_{\Xi,+}$ and with a non--trivial component of $\om_{u_0}(\xi)$ in $V_{\Xi,\alpha_i}$ for some (and thus every) $u_0$ covering $x_0$.
\end{itemize}

We have a priori information about $\Phi$ provided by the harmonic curvature $\kappa_H$ of $(\ba\to M,\om)$. The harmonic curvature has its values in $\fg_{\Xi,0}$--module $H^2(\fg_{\Xi,-},\fg)$ and each component $\mu$ of $\ka_H$ is represented by ordered pair $(\alpha_a, \alpha_b)$ meaning that the corresponding weight determining the component is obtained by the affine action of $s_{\alpha_a} s_{\alpha_b}$ on the highest root $\mu^\fg$ of $\fg$, where $s_{\alpha_i}$ denotes the simple reflexion along $\alpha_i$, see \ref{ApA} for detail explanation. We prove in the Proposition \ref{2.31} that each non--trivial
component $\mu$ of $\ka_H$ forces the set $\Phi$ to be a subset of the following set given by the results of [9]. 
\begin{itemize}
\item We denote by $I_\mu$ the set of roots $ \alpha_i \in \Xi$ that satisfy  $\langle s_{\alpha_a} s_{\alpha_b}(\mu^\fg+\delta)-\delta, \alpha_i \rangle =0$ for the component $\mu$ represented by $(\alpha_a, \alpha_b)$, where $\langle - , - \rangle $ denotes the Killing form and $\delta$ is the lowest form of $\fg$.
\end{itemize}
The reader can found possible sets $I_\mu$ in the tables in the \ref{ApC}.

\begin{example}
If $\fg=\frak{sl}(n+1,\mathbb{R}), n\geq 2$ and $\Xi=\{\alpha_1,\alpha_2\}$, then $(\alpha_1,\alpha_2)$ and $(\alpha_2,\alpha_1)$ are the only possible $\mu$ and $I_{(\alpha_1,\alpha_2)}=\emptyset$ and $I_{(\alpha_2,\alpha_1)}=\{\alpha_1\}$. Thus if $\kappa_H$  is non--trivial, then $\Phi\subset \{\alpha_1\}$ and if $\kappa_H$ has non--trivial component $(\alpha_1,\alpha_2)$, then $\Phi= \emptyset$.
\end{example}

\begin{itemize}
\item We denote by $\mathcal{J}$ the set consisting of all elements $s \in Z(G_0)$ such that there is $s$--symmetry of $(\ba\to M,\om)$ and define $$\Theta:=\{\alpha_i\in \Xi: \Ad_s|_{V_{\Xi,\alpha_i}}=\id {\rm \ for\ each\ }  s\in \mathcal{J}\}.$$ 
\end{itemize}

This means that the set $\Theta$ corresponds to the common eigenspace of all generalized symmetries of $(\ba\to M,\om)$ at $x_0$ with the eigenvalue $1$ in $T^{-1}_{x_0}M$. If trivial automorphisms are the only generalized symmetries, then $\Theta=\Xi$. Moreover, the non--trivial components of harmonic curvature of $(\ba\to M,\om)$ pose restrictions on possible $\mathcal{J}$ and $\Theta$. We summarize these restrictions in tables in the \ref{ApC}.

\begin{example}
Consider $\fg=\frak{sl}(n+1,\mathbb{R}), n\geq 2,$ and $\Xi=\{\alpha_1,\alpha_2\}$. If both of the components of $\kappa_H$ are non--trivial, then the only possible non--trivial generalized symmetry has eignevalues $j_1=1$ and $j_2=-1$. Then $\Theta=\{\alpha_1\}$ if there is a non--trivial generalized symmetry in $\A$ in this case. 

If only the component $(\alpha_1,\alpha_2)$ of $\kappa_H$ is non--trivial, then there can be a generalized symmetry in $\A$ with eigenvalues $j_1=-1$ and $j_2=1$ and  $\Theta=\{\alpha_2\}$ if there are no other types of generalized symmetries in $\mathcal{J}$ (up to trivial ones).
\end{example}

As mentioned above, we need to study several different parabolic subalgebras and we use the following concepts and notation for them:

For arbitrary subset $\Xi'$ of the set of simple restricted roots of $\fg$, we denote by $\fg_{\Xi',i}$ the grading of $\fg$ given by $\Xi'$--heights and we use the notation $\fg_{\Xi',-}$ and $\fp_{\Xi',+}$ for the negative and positive parts of the grading. We denote by $\fp_{\Xi'}:=\fg_{\Xi',0}\oplus \fp_{\Xi',+}$ the standard parabolic subalgebra of $\fg$, and by $\fq_{\Xi'}:=\fg_{\Xi',-}\oplus \fg_{\Xi',0}$ the standard opposite parabolic subalgebra of $\fg$ given by $\Xi'$. We assume that the corresponding subgroups $P_{\Xi'}$ and $Q_{\Xi'}$ with the Lie algebras $\fp_{\Xi'}$ and $\fq_{\Xi'}$ are the normalizers of $\fp_{\Xi'}$ and $\fq_{\Xi'}$ in $G$, respectively, and we denote $G_{\Xi',0}:=P_{\Xi'}\cap Q_{\Xi'}$. 

\begin{itemize}
\item We denote by $V_{\Xi',\gamma}$ the indecomposable $G_{\Xi',0}$--submodule of $\fg$ containing the root space $\fg_{\gamma}$.
\end{itemize}
Let us point out that $P$ is a subgroup of $P_\Xi$ containing  the component identity of $P_\Xi$ and $G_0$ is a subgroup of $G_{\Xi,0}$ containing the component identity of $G_{\Xi,0}$.

Having all these notation at hand, we can formulate the main result.

\begin{thm*}\label{main2}
Suppose the parabolic geometry $(\ba\to M,\om)$ of type $(G,P)$ is non--flat and consider arbitrary set of simple restricted roots $\Lambda$ satisfying
$$\Lambda\subset \Xi-\Phi-\Theta.$$ 
There is a (unique up to isomorphism) Cartan geometry $(\ba_{\Lambda}\to M,\om^{\Lambda})$ of type $(Q_{\Lambda},Q_{\Lambda}\cap P)$ such that:
\begin{enumerate}
\item $(\ba_{\Lambda}\to M,\om^{\Lambda})$ determines the same regular infinitesimal flag structure described in \cite[Sections 3.1.6,7]{parabook} as $(\ba\to M,\om)$ or the same finer underlying geometric structure described in \cite[Sections 3.1.15-16]{parabook} if $(\ba\to M,\om)$ is of projective or contact projective type.
\item $\A={\rm Aut}(\ba_{\Lambda},\om^{\Lambda})$.
\item The decomposition of $\om^{\Lambda}$ of the form $\om^{\Lambda}=\om^{\Lambda}_-+\om^{\Lambda}_0$ according to the values in $\fg_{\Lambda,-}\oplus \fg_{\Lambda,0}$ is $\A$--invariant.
\item There is $\A$--invariant distribution $VM$ on $M$ of the form
$$VM:=\ba_{\Lambda}\times_{(Q_{\Lambda}\cap P)} \fg_{\Lambda,0}/(\fg_{\Lambda,0}\cap \fp_\Xi).$$
\item The part $\om^{\Lambda}_-$ provides the $\A$--invariant complementary distribution $T^{\Lambda,-}M$ to $VM$ in $TM$ of the form $$T^{\Lambda,-}M:=\ba_{\Lambda}\times_{(Q_{\Lambda}\cap P)} \fg_{\Lambda,-}.$$ 
\item The part $\om^{\Lambda}_0$ defines  $\A$--invariant connection on $T^{\Lambda,-}M$.
\end{enumerate}
\end{thm*}
\begin{proof}
The  Proposition \ref{2.1} shows that the first two claims follow from Theorems \ref{main} and \ref{5.1}. In the case of projective and contact projective structures, the claim (1) follows from the existence and uniqueness of invariant (partial) Weyl connection, see \cite{disertace,GZ}. The rest follows from the Proposition \ref{5.3}.
\end{proof}
\begin{rem*}
In fact, it is possible to choose $\Lambda$ larger than $\Xi-\Phi-\Theta$ due to the Proposition \ref{2.1}, but it is not easy to relate all possible choices with generalized symmetries and compute the bound for possible sets $\Lambda$.
\end{rem*}

In fact, if there is no non--trivial generalized symmetry in $\A$, then $\Lambda\subset\emptyset$ and the claim in the Theorem \ref{main2} is trivial. The example in the \ref{ApB} shows that the set $\Theta$ is essential for the claim to hold. On the other hand, for most of the types of parabolic geometries, the classification results presented in the \ref{ApC} ensures that $\Lambda\neq \emptyset$ when there is a non--trivial generalized symmetry in $\A$.

\begin{example}
Consider $\fg=\frak{sl}(n+1,\mathbb{R}), n\geq 2,$ $\Xi=\{\alpha_1,\alpha_2\}$. Assume $\kappa_H\neq 0$ and assume there is a non--trivial generalized symmetry in $\A$. Then there are the following non--trivial choices of $\Lambda$ that are always possible:
\begin{enumerate}
\item If both possible components of $\kappa_H$ are non--trivial, then we can choose $\Lambda=\{\alpha_2\}=\Xi-\emptyset-\{\alpha_1\}.$
\item If only the component $(\alpha_1,\alpha_2)$ of $\kappa_H$ is non--trivial, then for generic $\mathcal{J}$, we can choose 
$\Lambda=\{\alpha_1,\alpha_2\}=\Xi-\emptyset-\emptyset.$ 
However, if all $s \in \mathcal{J}$ have eigenvalue $1$ on $\alpha_i$, $i=1,2$, then we can choose
$\Lambda=\{\alpha_j\}=\Xi-\emptyset-\{\alpha_i\}$ 
for $j=1,2$ and $j \neq i$.
\item If only the component $(\alpha_2,\alpha_1)$ of $\kappa_H$ is non--trivial, then we can choose $\Lambda=\{\alpha_2\}=\Xi-\{\alpha_1\}-\{\alpha_1\}$ or $\Lambda=\{\alpha_2\}=\Xi-\{\alpha_1\}-\emptyset.$
\end{enumerate}
\end{example}


Let us pose here the following application of the Theorem to demonstrate here the power of our reductions.

\begin{cor*}
If there is a non--trivial generalized symmetry of $(\ba\to M,\om)$ such that $\kappa_H\neq 0$ and $\fg_{\Xi,-1}$ is indecomposable $G_0$--module, then 
\begin{enumerate}
\item $\Lambda$ can be chosen as $\Xi$, 
\item the parabolic geometry $(\ba_{\Lambda}\to M, \om^{\Lambda})$ is of type $(Q_{\Lambda},G_0)$, 
\item there is an $\A$--invariant linear connection on $TM$, and
\item generalized symmetries are the generalized affine symmetries in the sense of \cite{kowalski}.
\end{enumerate}
\end{cor*}

There are further geometric properties of our geometries provided by results of \cite{Cap-correspondence} about correspondence and twistor spaces.  Let us introduce the following notation:
\begin{itemize}
\item We denote by $\Psi$ the maximal subset of $\Xi$ such that the harmonic curvature does not have entries in $\fg_{\Xi-\Psi,0}$.
\end{itemize}
The set $\Psi$ characterizes possible (local) twistor spaces for the parabolic geometries of type $(G,P)$ by results of \cite{Cap-correspondence}. Let us point out that the set $\Psi$ can be computed explicitly from $\kappa_H$, see \ref{ApA}. 

The generalized symmetries can provide a sufficient condition for global existence of the twistor space, see Theorem \ref{4.5}. Let us present here only the explicit results of the Theorem \ref{4.5} for the example of the generalized path geometries.

\begin{example}
Consider $\fg=\frak{sl}(n+1,\mathbb{R}), n\geq 2$ and $\Xi=\{\alpha_1,\alpha_2\}$. Assume $\kappa_H\neq 0$ and assume there is a non--trivial generalized symmetry in $\A$. Then there are the following cases:
\begin{enumerate}
\item If both possible components of $\kappa_H$ are non--trivial, then $\Psi=\emptyset$.
\item If only the component $(\alpha_1,\alpha_2)$ of $\kappa_H$ is non--trivial, then $\Psi=\{\alpha_2\}$. Then if there is a generalized symmetry with eigenvalues $j_1=-1$ and  $j_2=1$, then the corresponding parabolic geometry $(\ba\to M,\om)$ is (a covering of) a correspondence space to the projective class of the canonical connection of affine symmetric space.
\item If only the component $(\alpha_2,\alpha_1)$ of $\kappa_H$ is non--trivial, then $\Psi=\{\alpha_1\}$. Then if there is a generalized symmetry with eigenvalues $j_1=1$ and  $j_2=-1$, then the corresponding parabolic geometry $(\ba\to M,\om)$ is (a covering of) a correspondence space to invariant para--quaternionic structure on affine symmetric space.
\end{enumerate}
\end{example}

{\bf
Outline of the structure of the article:}
In the second section, we recall the description of homogeneous parabolic geometries and investigate the possibilities, how to reduce a homogeneous parabolic geometry. In particular, we obtain an algebraic condition for the existence of a reduction, see Theorem \ref{5.1}.

In the third section, we construct a significant homogeneous parabolic geometry with better algebraic properties and show that it is enough to investigate these algebraic conditions to discuss the existence of a reduction, see Theorem \ref{extension}.

In the fourth section, we show how the generalized symmetries of arbitrary homogeneous parabolic geometry are related to generalized symmetries of the significant homogeneous parabolic geometry constructed in the third section, and show how can the existence of generalized symmetries (in the non--flat case) imply the existence of non--trivial reductions, see Theorem \ref{main}.

In the fifth section, we describe a particular class of Weyl structures  naturally compatible with these reductions, see Proposition \ref{lamweyl}. Clearly, Weyl connections corresponding to this class provide tools for further computations on these geometries.

In the sixth section, we discuss the correspondence and twistor spaces related to the homogeneous parabolic geometry. We investigate, when the twistor space can be constructed globally, and we discuss the compatibility of this construction with generalized symmetries, see Theorem \ref{4.5}.

In the seventh section, we investigate conditions on $\kappa_H$ under which, there is generically an invariant Weyl connection on the non--flat homogeneous parabolic geometry with a non--trivial generalized symmetry. It turns out that such parabolic geometries are closely related to the usual symmetric spaces, see Theorem \ref{7.1}.

At the end of the article, there are three appendixes attached. The \ref{ApA} recalls the details about the harmonic curvature $\kappa_H$ and the notation derived from it. The \ref{ApB} contains a particular example of a homogeneous parabolic geometry on which we demonstrate our theory. The \ref{ApC} then contains the complete classification of types of parabolic geometries that can satisfy $\kappa_H\neq 0$ and admit non--trivial generalized symmetries in $\A$ for $\fg$ simple.

\section{Extensions, reductions and homogeneous parabolic geometries} \label{dva}

Let us recall that each homogeneous parabolic geometry $(\ba \to M, \om)$ of type $(G,P)$ admits a completely algebraic description, and we follow here the concepts and notation of \cite[Sections 1.5.]{parabook}.

\begin{def*}
Let $K$ be a Lie group and let 
$H$ be a closed subgroup of $K$. Let $\iota: H\to P$ be a Lie group homomorphism and
$\alpha: \fk\to \fg$ be a linear map satisfying:
\begin{enumerate}
\item $\alpha|_\fh=d\iota$,
\item $\alpha$ induces a linear isomorphism of $\fk/\fh$ and $\fg/\fp$,
\item $\Ad_{\iota(h)}\circ \alpha=\alpha \circ \Ad_h$ for all $h\in H$.
\end{enumerate}
Then we say that $(\alpha,\iota)$ is the \emph{extension of $(K,H)$ to $(G,P)$}.
\end{def*}

If $(\alpha,\iota)$ is an extension of $(K,H)$ to $(G,P)$, then there is $K$--invariant Cartan connection $\om_\alpha$ of type $(G,P)$ on $K\times_{\iota(H)}P\to K/H$ induced by the $\alpha$--image of the Maurer--Cartan form on $K$.

\begin{def*}
We say that the parabolic geometry $(K\times_{\iota(H)}P\to K/H,\om_\alpha)$ of type $(G,P)$ is \emph{given by extension $(\alpha,\iota)$ of $(K,H)$ to $(G,P)$}.
\end{def*}

One of the crucial results about homogeneous parabolic geometries \cite[Theorem 1.5.15.]{parabook} is that for any subgroup $K$ of $\A$ acting transitively on $M$ and for any point $u_0\in \ba$, there is an extension $(\alpha,\iota)$ of $(K,H)$ to $(G,P)$ with injective $\iota$ such that the map $K\times_{\iota(H)} P\to \ba$ defined as $\lz k,p\pz\mapsto k(u_0)p$ is isomorphism of parabolic geometries $(K\times_{\iota(H)} P\to K/H,\om_\alpha)$ and $(\ba\to M,\om)$. Clearly, the extension $(\alpha,\iota)$ depends on the choice of $u_0$ and if we choose $u_0p^{-1}$ instead of $u_0$, then we end up with the extension $(\Ad_p\circ \alpha,\conj_p\circ \iota)$ of $(K,H)$ to $(G,P)$.

\begin{def*}
We say that $(\ba\to M,\om)$ is \emph{given by the extension $(\alpha,i)$ of $(K,H)$ to $(G,P)$ at $u_0$} if there is an inclusion of $K$ into $\A$ such that the map $\lz k,p\pz\mapsto k(u_0)p$ induced by this inclusion is an isomorphism of $(K\times_{\iota(H)} P\to K/H,\om_\alpha)$ and $(\ba\to M,\om)$.
\end{def*}

Moreover, each extension $(\alpha,\iota)$ of $(K,H)$ to $(G,P)$ defines a functor from category of Cartan geometries of type $(K,H)$ to category of Cartan geometries of type $(G,P)$.

We would like to discuss here reductions of the parabolic geometry $(\ba\to M,\om)$ by means of the theory of holonomy reductions for parabolic geometries introduced in \cite{Cap-holonomy}. 
Let us recall \cite[Definition 2.2]{Cap-holonomy} which defines a holonomy reduction of $(\ba\to M,\om)$ of $G$--type $\mathcal{O}$ as a parallel section of $\ba\times_P \mathcal{O}$  w.r.t. the natural (non--linear) connection on $\ba\times_P \mathcal{O}$ induced by the Cartan connection $\om$, where $\mathcal{O}$ is a $G$--homogeneous space.
There is the necessary condition for the existence of a holonomy reduction following from \cite[Theorem 2.6]{Cap-holonomy} that can be expressed via extension functors as follows. 

\begin{prop*}
Let $\Xi'\subset \Xi$ be a set of restricted roots and let $\mathcal{F}_{\Xi'\subset \Xi}$ be the extension functor induced by the inclusion $Q_{\Xi'}\subset G$. If there is a holonomy reduction $\tau$ of $(\ba\to M,\om)$ of $G$--type $G/Q_{\Xi'}$, then $(\ba\to M,\om)$ restricted to the submanifold $p \circ \tau^{-1}(PeQ_{\Xi'})$ of $M$ is contained in the image of $\mathcal{F}_{\Xi'\subset \Xi}$, where $p$ is the projection $\ba\to M$.
\end{prop*}
\begin{proof}
The claim (ii) of \cite[Theorem 2.6]{Cap-holonomy} can be rephrased such that the claimed property is necessary for the existence of the holonomy reduction $\tau$. Indeed, the Cartan bundle of the reduced Cartan geometry of type $(Q_{\Xi'},Q_{\Xi'}\cap P)$ consists of points $u_0\in \ba$ such that  $\tau(u_0)\in Q_{\Xi'}$.
\end{proof}

The main difficulty is that the parabolic geometries of type $(G,P)$ in the image of the functor $\mathcal{F}_{\Xi'\subset \Xi}$ are not normal in general. Therefore, it makes no sense to reduce the normal parabolic geometries. 

On the other hand, we are only interested in holonomy reductions compatible with the automorphisms of $(\ba\to M,\om)$. If we choose $K\subset \A$ acting transitively on $M$, then we restrict our attention to $K$--invariant holonomy reductions. Then $M$ consists of points of the same $P$--type w.r.t. to $K$--invariant holonomy reduction and we have to find a parabolic geometry of type $(G,P)$, which describes the same infinitesimal flag structure as $(\ba\to M,\om)$ and which is in the image of the functor $\mathcal{F}_{\Xi'\subset \Xi}$. The obvious necessary condition for the existence of such homogeneous parabolic geometry  in the image of $\mathcal{F}_{\Xi'\subset \Xi}$ is  $$\iota(H)\subset Q_{\Xi'}\cap P.$$ 
We will prove that this condition is in fact sufficient for the existence of the holonomy reduction.

\begin{thm*}\label{5.1}
Let $(\ba\to M,\om)$ be given by the extension $(\alpha,\iota)$ of $(K,H)$ to $(G,P)$ at $u_0$ and suppose $\Xi'$ is subset of $\Xi$  such that $\iota(H)\subset Q_{\Xi'}\cap P$. Consider the decomposition
$\alpha =\alpha_{\Xi',-}+ \alpha_{\Xi',0} + \alpha_{\Xi',+}$ 
according to the values in
$\fg=\fg_{\Xi',-}\oplus \fg_{\Xi',0} \oplus \fp_{\Xi',+}.$
There is the parabolic geometry $$(\ba\to M, \om_{\alpha_{\Xi',-}+ \alpha_{\Xi',0}})$$ given at $u_0$ by extension $(\alpha_{\Xi',-}+ \alpha_{\Xi',0},\iota)$ of $(K,H)$ to $(G,P)$ providing the same 
underlying  infinitesimal flag structure as $(\ba\to M,\om)$, and the following claims hold:
\begin{enumerate}
\item  The section $\tau$ of $\ba\times_P G/Q_{\Xi'}$ given by the projection onto the quotient of $Ku_0\subset \ba\times_P G$ is a holonomy reduction of $(\ba\to M, \om_{\alpha_{\Xi',-}+ \alpha_{\Xi',0}})$ of $G$--type $G/Q_{\Xi'}$.
\item  All points of $M$ have the same $P$--type with respect to $\tau$. Let us denote $$\ba_{\Xi'}(u_0):=Ku_0(Q_{\Xi'}\cap P)=\tau^{-1}(eQ_{\Xi'}),$$  $$\om^{\Xi'}:=j_{eQ_{\Xi'}}^*(\om_{\alpha_{\Xi',-}+ \alpha_{\Xi',0}}),$$
where $j_{eQ_{\Xi'}}$ is the natural inclusion $\ba_{\Xi'}(u_0)\hookrightarrow \ba$. Then $$(\ba_{\Xi'}(u_0)\to M,\om^{\Xi'})$$ is a homogeneous Cartan geometry of type $(Q_{\Xi'},Q_{\Xi'}\cap P)$.
\item It holds $$(\ba\to M, \om_{\alpha_{\Xi',-}+ \alpha_{\Xi',0}})=\mathcal{F}_{\Xi'\subset \Xi}(\ba_{\Xi'}(u_0)\to M,\om^{\Xi'}).$$
\end{enumerate}
Suppose $p \in P$ is such that $\conj_p^{-1}\iota(H)\subset Q_{\Xi'}\cap P$. Then the section of $\ba\times_P G/Q_{\Xi'}$ given by the projection onto the quotient of $Ku_0p\subset \ba\times_P G$ is exactly the section $\tau$. Thus the holonomy reduction does not depend on the choice of $u_0$.
\end{thm*}
\begin{proof}
Since $\iota(H)\subset Q_{\Xi'}\cap P$, the projection onto the quotient of $Ku_0\subset \ba\times_P G$ is a correctly defined smooth section of $\ba\times_P G/Q_{\Xi'}$. Moreover, the corresponding function $\ba\to G/Q_{\Xi'}$ has its values in the class $PeQ_{\Xi'}$ in $P\backslash G/Q_{\Xi'}$. Thus if $\tau$ is a holonomy reduction, then all points of $M$ have the same $P$--type with respect to $\tau$.

Further, $\iota(H)\subset Q_{\Xi'}\cap P$ implies that the decomposition of $\alpha$ is $H$--invariant, and thus $(\alpha_{\Xi',-}+ \alpha_{\Xi',0},\iota)$ is extension of $(K,H)$ to $(G,P)$. By definition, the parabolic geometry $(\ba\to M, \om_{\alpha_{\Xi',-}+ \alpha_{\Xi',0}})$ given by the extension $(\alpha_{\Xi',-}+ \alpha_{\Xi',0},\iota)$ at $u_0$ induces the same filtration of $TM$ and the same reduction to $G_0$,  i.e., the same underlying infinitesimal flag structure.
It is clear from the construction of the connection on $\ba \times_P G/Q_{\Xi'}$ that the section $\tau$ is parallel, i.e. it is a holonomy reduction. The remaining claims from the list follow from \cite[Theorem 2.6]{Cap-holonomy} immediately.

Finally, it is a simple computation to show that the projection onto the quotient of $Ku_0p\subset \ba\times_P G$ is exactly $\tau$ for $p^{-1}Hp\subset Q_{\Xi'}\cap P$.
\end{proof}

In the setting of the Theorem \ref{5.1}, we can interpret the underlying geometric structure for Cartan geometries of type $(Q_{\Xi'},Q_{\Xi'}\cap P)$ as follows.
Let us fix the reduction $(\ba_{\Xi'}\to M,\om^{\Xi'})$ given by the fixed choice of $u_0$ such that $\iota(H)\subset Q_{\Xi'}\cap P$.

\begin{prop*} \label{5.3}
Let $\om^{\Xi'}_-+\om^{\Xi'}_0$ be the decomposition of $\om^{\Xi'}$ according to the values in $\fg_{\Xi',-}\oplus \fg_{\Xi',0}$. Denote 
$$VM:=\ba_{\Xi'}\times_{(Q_{\Xi'}\cap P)} \fg_{\Xi',0}/(\fg_{\Xi',0}\cap \fp_\Xi).$$
\begin{enumerate}
\item The part $\om^{\Xi'}_-$ provides the $K$--invariant complement of $K$--invariant subbundle $VM$ in $TM$ of the form $$T^{\Xi',-}M:=\ba_{\Xi'}\times_{(Q_{\Xi'}\cap P)} \fg_{\Xi',-}.$$ 
\item The part $\om^{\Xi'}_0$ defines a $K$--invariant connection on $T^{\Xi',-}M$.
\end{enumerate}
\end{prop*}
\begin{proof}
The claims follow directly from the inclusion $\iota(H)\subset Q_{\Xi'}\cap P$ and the reductivity of the pair $(Q_{\Xi'},G_{\Xi',0})$.
\end{proof}

In general, we cannot say anything about the image $\iota(H)$ in $P$. However, we will prove in the following sections, that existence of generalized symmetries provides us interesting information about the image $\iota(H)$ in $P$.

\section{The fundamental construction} \label{tri}
The main aim now is to study conditions under which, there is non--empty admissible set $\Xi'$ for which the assumptions of the Theorem \ref{5.1} are satisfied. We know from above that we have to study the condition $\iota(H)\subset Q_{\Xi'}\cap P$. In particular, we no longer need to work with the parabolic geometry $(\ba\to M,\om)$ of type $(G,P)$. In fact, we will work in a more convenient algebraic setting. Let us introduce the necessary notation:
\begin{itemize}
 \item We denote by $\tG$ the linear algebraic Lie group of automorphisms of $\fg$ and we identify its Lie algebra of derivations of $\fg$ with $\fg$.
\item We denote by $\tP_\Xi$ the normalizer of $\fp_\Xi$ in $\tilde G$ under the natural action of the automorphisms group on $\fg$, i.e., $\tP_\Xi$ is a linear algebraic parabolic subgroup of $\tilde G$ with Lie algebra $\fp_\Xi$.
\item For each subset $\Xi'$ of the set of simple restricted roots, we denote by $\tilde Q_{\Xi'}$ and $\tP_{\Xi'}$ the normalizers of $\fq_{\Xi'}$ and $\fp_{\Xi'}$ in $\tilde G$, respectively. Then $\tG_{\Xi',0}:=\tilde Q_{\Xi'}\cap \tP_{\Xi'}$ is the (algebraic) Levi subgroup of $\tP_{\Xi'}$.
\end{itemize}

Let us point out that contrary to $Z(G_0)$, the center $Z(\tG_{\Xi,0})$ posses no restrictions on possible eigenvalues of generalized symmetries apart the obvious ones, see \cite[Proposition 3.2]{GZ2}.

\begin{thm*} \label{extension}
Let $(\alpha,\iota)$ be extension of $(K,H)$ to $(G,P)$ giving $(\ba\to M,\om)$ at $u_0$. There is an element $p\in  \exp(\fp_{\Xi,+})$, Lie groups $\tilde H \subset \tilde K$ and extension $(\ta,\ti)$ of $(\tK,\tH)$ to $(\tG,\tP_\Xi)$ such that the following is satisfied:
\begin{enumerate}
\item It holds $$\tH=\{h\in \tP_\Xi: \Ad_h(\ta(\tfk))\subset \ta(\tfk),\ \Ad_h.\kappa(u_0p^{-1})=\kappa(u_0p^{-1})\},$$ where $.$ is the natural action of $P$ on $\wedge^2 (\fg/\fp_\Xi)^*\otimes \fg$. In other words, $\ti$ is the inclusion $\tH\subset \tP_\Xi$.
\item It holds 
$\Ad\circ \conj_p\circ \iota(H)\subset \tH$ and
$\Ad_p\circ d\iota(\fh)\cap \fp_{\Xi,+}= \tfh\cap\fp_{\Xi,+}.$
In particular, if $\tH\subset \tilde Q_{\Xi'}\cap \tP_\Xi$, then $\conj_p\circ \iota(H)\subset Q_{\Xi'}\cap P.$
\item There is algebraic Levi decomposition $\tH=\tHn \cdot \exp(\tfn)$ for some reductive Lie subgroup $\tHn$ contained in $\tG_{\Xi,0}$ and for the maximal normal Lie subgroup consisting of unipotent elements of $\tP_\Xi$ of the form $\exp(\tfn)$ for some nilpotent subalgebra $\tfn$ of $\fp_\Xi$.
In particular, $$\tfh\cap\fp_{\Xi,+}=\tfn\cap \fp_{\Xi,+}.$$
\item The space $\tK/\tH$ is the simply connected covering of $K/H$.
\end{enumerate}
\end{thm*}
\begin{rem*}
We construct the group $\tK$ and the map $\ta$ explicitly using an induction in the proof  and we will use these objects in the next section. 
\end{rem*}
\begin{proof}
Assume there is an extension $(\ta,\ti)$ satisfying conditions (1) and (2). Since $\tH$ is linear algebraic Lie subgroup of the linear algebraic Lie group $\tP_\Xi$, the maximal normal subgroup of $\tH$ consisting of unipotent elements of $\tP_\Xi$ is a connected Lie subgroup of $\tH$ and thus has form $\exp(\tfn)$ for some nilpotent subalgebra $\tfn$ of $\fp_\Xi$. Therefore there is algebraic Levi decomposition $\tH=\tHn \cdot \exp(\tfn)$ from \cite[Chapter 1, Theorem 6.7.]{Lie-alg-3}, where $\tHn$ is maximal reductive subgroup of $\tP_\Xi$ that is contained in $\tH$. It follows from \cite[Chapter 1, Theorem 6.8.]{Lie-alg-3}, that there is $p\in \exp(\fp_{\Xi,+})$ such that $\conj_p (\tHn)\subset \tG_{\Xi,0}$. Therefore, if we start our construction with extension $(\Ad_p\circ \alpha,\conj_p\circ i)$ giving $(\ba\to M,\om)$ at $u_0p^{-1}$,  then the resulting extension  $(\ta,\ti)$ satisfies all the conditions (1)--(3).

The construction of the extension $(\ta,\ti)$ satisfying (1) and (2) will be done inductively in finitely many steps. We start with $K^0:=K$ and $(\alpha^0,\iota^0):=(\alpha,\iota)$. We construct extension $(\alpha^{j+1},\iota^{j+1})$ of $(K^{j+1},H^{j+1})$ to $(\tG,\tP_\Xi)$ from the extension $(\alpha^{j},\iota^{j})$ of $(K^{j},H^{j})$ as follows: We define 
$$N_{\tP_\Xi}(\alpha^j):=\{h\in \tP_\Xi: \Ad_h(\alpha^j(\fk^j))\subset \alpha^j(\fk^j), \Ad_h.\kappa(u_0)=\kappa(u_0)\}.$$ 
We know from \cite[Lemma 2.2]{GZ2} that $N_{\tP_\Xi}(\alpha^j)$ acts by automorphisms on the Lie algebra $\fk^j$. Thus we can form the semidirect products of Lie groups $K^{jc}\rtimes N_{\tP_\Xi}(\alpha^j)$ and $H^{jc}\rtimes N_{\tP_\Xi}(\alpha^j)$, where $K^{jc}$ is connected simply connected Lie group with Lie algebra $\fk^j$ and $H^{jc}$ is such that $K^{jc}/H^{jc}$ is simply connected covering of $K^j/H^j$. 
We define 
$$\iota^{j+1}: H^{jc}\rtimes N_{\tP_\Xi}(\alpha^j)\to \tP_\Xi,\ \ \iota^{j+1}(h,q)=\Ad_h\circ \Ad_q.$$ 
This is clearly Lie algebra homomorphism with the kernel generated by elements $(h,\Ad_h^{-1})$ for $h\in H^{jc}$. If we show that $Ker(\iota^{j+1})$ is normal subgroup of $K^{jc}\rtimes N_{\tP_\Xi}(\alpha^j)$, then we can define 
$$K^{j+1}:=K^{jc}\rtimes N_{\tP_\Xi}(\alpha^j)/Ker(\iota^{j+1}),\ \ H^{j+1}:=H^{jc}\rtimes N_{\tP_\Xi}(\alpha^j)/Ker(\iota^{j+1})$$ and $$\alpha^{j+1}=\alpha^{j}+\id.$$ Since $\alpha^{j+1}|_{\fk^j}=\alpha^j$ and $\alpha^{j+1}|_{\fh^{j+1}}=d\iota^{j+1}$, it is clear that $(\alpha^{j+1},\iota^{j+1})$ is extension of $(K^{j+1},H^{j+1})$ to $(\tG,\tP_\Xi)$.

Thus let us prove $(k,l)(h,\Ad_h^{-1})(k,l)^{-1}\in Ker(\iota^{j+1})$ for all $h\in H^{jc}$ and all $(k,l)\in K^{jc}\rtimes N_{\tP_\Xi}(\alpha^j)$. 
It holds $$(k,l)(h,\Ad_h^{-1})(k,l)^{-1}=(kl(h\Ad_h^{-1}(l^{-1}(k^{-1}))),l\Ad_h^{-1}l^{-1})$$ and 
$$\Ad_{kl(h\Ad_h^{-1}(l^{-1}(k^{-1})))}=\Ad_k\Ad_l\Ad_h\Ad_h^{-1}\Ad_l^{-1}\Ad_k^{-1}\Ad_l\Ad_h^{-1}\Ad_l^{-1}
=\Ad_l\Ad_h^{-1}\Ad_l^{-1},$$  i.e., it is enough to show $kl(h\Ad_h^{-1}(l^{-1}(k^{-1})))\in H^{jc}$. This is clear, because $\conj_h=\Ad_h$ as automorphism of $K^{jc}$ for $h\in H^{jc}$ and thus $kl(h\Ad_h^{-1}(l^{-1}(k^{-1})))l(h^{-1})=e$.

If $\fk^j=\fk^{j+1}$, then $(\ta,\ti):=(\alpha^{j+1},\iota^{j+1})$ is extension of $(\tK,\tH):=(K^{j+1},H^{j+1})$ to $(\tG,\tP_\Xi)$ satisfying (1) by definition of $\tH$, and (2) follows from \cite[Lemma 3.6.]{GZ2}. Since $\fp_\Xi$ is finite dimensional, $\fk^j=\fk^{j+1}$ holds in finite number of steps. By construction, (4) holds, too.
\end{proof}

Therefore it suffices to work with the extension $(\ta,\ti)$ of $(\tK,\tH)$ to $(\tG,\tP_\Xi)$ from the above Theorem and prove $\tH\subset \tilde Q_{\Xi'}\cap \tP_\Xi$ in order to obtain reduction from the Theorem \ref{5.1}.

\section{The reduction given by generalized symmetries}
 Let us return to the role of generalized symmetries  in studying of  holonomy reductions. Firstly, let us summarize some more definitions and notation that will simplify our work in the future.
In fact, we generalize the definitions of the objects $\Phi,\mathcal{J},\Theta$ and $\Lambda$ from the Introduction for the data provided by the extension $(\ta,\ti)$ from the Theorem \ref{extension}.

\begin{itemize}
\item  We denote by $\Phi(\fk)$ the set of all $\alpha_i\in \Xi$ such that there is element of $\tfn\cap \fp_{\Xi,+}$ with non--trivial component in $V_{\Xi,\alpha_i}$.
\item We denote $\mathcal{J}(\fk):=\tH\cap Z(\tilde G_{\Xi,0})$.
\item We define $\Theta(\fk):=\{\alpha_i\in \Xi: \Ad_s|_{V_{\Xi,\alpha_i}}=\id {\rm \ for\ each\ }  s\in \mathcal{J}(\fk)\}.$ 
\item We denote $\Lambda(\fk):=\Xi-\Phi(\fk)-\Theta(\fk)$.
\end{itemize}

Let us point out that these objects depend on the choice of extension $(\alpha, \iota)$ of $(K,H)$ to $(G,P)$ on which we apply the construction from the Theorem \ref{extension}. We indicate this by the argument $\fk$ in the notation. The relation to the objects $\Phi,\mathcal{J},\Theta$ and $\Lambda$ from the Introduction is the following:

\begin{prop*}\label{2.1}
If $K=\A$, then $\Phi(\fk)=\Phi$, $\mathcal{J}\subset \mathcal{J}(\fk)$ and thus $\Lambda\subset \Xi-\Phi-\Theta\subset \Lambda(\fk).$
\end{prop*}
\begin{rem*}
In general, there can be elements in $\mathcal{J}(\fk)$ that are not contained in $\Ad(Z(G_0))$. Therefore by considering $\Lambda(\fk)$ instead of $\Lambda$, we can obtain a stronger reduction from the Theorem \ref{5.1}. 
 Similarly, the automorphism group of the parabolic geometry given by extension $(\ta,\ti)$ can be larger than $\tK$ for the choice $K=\A$, which can again provide a stronger reduction from the Theorem \ref{5.1}. These are the main reasons, why we generally consider $K\neq \A$.
\end{rem*}
\begin{proof}
If $K=\A$, then $\Phi(\fk)=\Phi$ follows from the statement (2) of the Theorem \ref{extension} and Proposition \ref{1.4} and $\mathcal{J}\subset \mathcal{J}(\fk)$ follows from the statement (3) of the Theorem \ref{extension}. Thus $\Xi-\Phi-\Theta\subset \Lambda(\fk)$ follows from the definition of $\Lambda(\fk)$. 
\end{proof}

Let us now show that the set $\mathcal{J}(\fk)$ provides a fundamental information on $\tH$, and simultaneously explain the role of the sets $\Phi(\fk)$, see \ref{ApA}  for the further notation.

\begin{prop*}\label{2.31}
Let $(\ta,\ti)$ be extension of $(\tK,\tH)$ to $(\tG,\tP_\Xi)$ from the Theorem \ref{extension}. Denote by $\fm$ the intersection of all $1$--eigenspaces in $\fp_{\Xi,+}$ for all $s\in \mathcal{J}(\fk)$. Then the following claims hold:

\begin{enumerate}
\item For each $\Xi'\subset \Xi$, it holds  $\tfh\cap \fp_{\Xi-\Xi',+}\subset \fp_{\Phi(\fk)-\Xi',+}$ and $\tfh\cap \fp_{\Xi-\Phi(\fk),+}=0$. In particular, 
$$\tfh\subset \fg_{\Xi,0}\oplus \fm+\fp_{\Phi(\fk)-\Theta(\fk),+}.$$
\item The set $\Phi(\fk)$ is a subset of each set $I_\mu$ corresponding to each of the non--trivial components $\mu$ of $\kappa_H$.
\item If  the parabolic geometry $(\ba\to M,\om)$ is non--flat, then there is basis of $\tfh\cap \fp_{\Xi-\Theta(\fk),+}$ such that there is decomposition $Z_0+[Z_0,R]+\dots$ of any basis element with the following properties:
\begin{itemize}
\item Each basis element is contained in an eigenspace for single eigenvalue for each $s\in \mathcal{J}(\fk).$
\item $Z_0\in V_{\Xi-\Theta(\fk),\delta}\subset  \fp_{\Phi(\fk),+}\cap  \fg_{\Lambda(k),0}$ for some restricted root $\delta$.
\item $R\in \fm\cap  \fp_{\Lambda(\fk),+} \cap \fg_{\Xi-\Theta(\fk),j'}$ for some $j'>0$.
\item $\dots$ are contained in grading components of $ \fp_{\Lambda(\fk),+}$ of $\Xi-\Theta(\fk)$--height greater then $j'$.
\end{itemize}
\item If the parabolic geometry $(\ba\to M,\om)$ is non--flat, then $$[\fm\cap \fg_{\Lambda(\fk),0}\cap \fp_{\Xi-\Theta(\fk),+},\fm\cap \fp_{\Lambda(\fk),+}]=0.$$
\end{enumerate}
\end{prop*}
\begin{proof}
The Lie algebra $\tfh$ splits simultaneously w.r.t. eigenvalues of all $s\in \mathcal{J}(\fk)$ and therefore to show $\tfh\subset \fg_{\Xi,0}\oplus \fm+\fp_{\Phi(\fk)-\Theta(\fk),+}$, it suffices to show $\tfh\cap \fp_{\Xi-\Theta(\fk),+}\subset \fp_{\Phi(\fk)-\Theta(\fk),+}$. 
We will show general results for arbitrary $\Xi'$ and the results for $\Theta(\fk)$ then are a special case.

So let $Z\in \tfh\cap \fp_{\Xi-\Xi',+}$ be arbitrary. We discuss step by step components of $Z$ in $\fp_{\Xi-\Xi',+}$ w.r.t. $\Phi(\fk)-\Xi'$--height and $\Xi-\Phi(\fk)-\Xi'$--height. Consider the decomposition of $Z$ of the form $Z=Z_{+0}+Z_{0+}+Z_{++}$, where $Z_{+0}$ sits in $\fp_{\Phi(\fk)-\Xi',+}\cap  \fg_{\Xi-\Phi(\fk)-\Xi',0}$, $Z_{0+}$ sits in $\fg_{\Phi(\fk)-\Xi',0}\cap  \fp_{\Xi-\Phi(\fk)-\Xi',+}$ and $Z_{++}$ sits in $\fp_{\Phi(\fk)-\Xi',+}\cap  \fp_{\Xi-\Phi(\fk)-\Xi',+}$.

Suppose $(Z_{+0})_j$ is the component of $Z_{+0}$ in $\fg_{\Xi-\Xi',j}$, where $j$ is lowest such that $(Z_{+0})_j\neq 0$. 
Similarly, let $(Z_{0+})_{l}$ be the component of $Z_{0+}$ in $\fg_{\Xi-\Xi',l}$,  where $l$ is lowest such that $(Z_{0+})_{l}\neq 0$. We show that the existence of non--trivial $(Z_{0+})_{l}$ leads to the contradiction with the definition of $\Phi(\fk)-\Xi'$. We divide the proof into two cases:

(a) If $l<j$, then since $\tilde \alpha(\tilde \fk/\tilde \fh)=\fg/\fp_\Xi$, it follows from \cite[Proposition 3.1.2]{parabook} that there is a sequence $X_1,\dots,X_i\in \tfk$ such that the element $[\ta(X_1),\cdots,[\ta(X_i),Z]\cdots]$ in $ \tfh\cap \fp_{\Xi-\Xi',+}$ has a non--trivial component in $\fg_{\Phi(\fk)-\Xi',0}\cap \fg_{\Xi,1}\cap \fg_{\Xi',0}$, which is contradiction with the definition of $\Phi(\fk)-\Xi'$. 

(b) If $l\geq j$, then there is $X\in \tfk$ such that $\ta(X)=X_-+S$ for suitable $S \in \fp_\Xi$ and $X_-\in \fg_{\Phi(\fk)-\Xi',-}\cap  \fg_{\Xi-\Phi(\fk)-\Xi',0}\cap \fg_{\Xi-\Xi',-j}$
 such that $(X_-,[X_-,(Z_{+0})_j],(Z_{+0})_j)$ forms an $\frak{sl}(2)$--triple. The existence of the   $\frak{sl}(2)$--triple follows from \cite[Lemma 10.18 and the proof of Corollary 10.22]{Kn96}.
Then $[X_-,(Z_{0+})_{l}]$ has positive $\Xi-\Phi(\fk)-\Xi'$--height but negative $\Phi(\fk)-\Xi'$--height and thus $[X_-,(Z_{0+})_{l}]=0$.
 Therefore the component of $[\ta(X),Z]\in \tfh$ in $\fg_{\Xi,0}$ is of the form $[X_-,(Z_{+0})_j]$. If $[[X_-,(Z_{+0})_j],(Z_{0+})_{l}]=2(Z_{0+})_{l}$, then \cite[Theorem 10.10]{Kn96} (after complexification) implies that $(Z_{0+})_{l}$ is contained in the image of $\ad_{(Z_{+0})_j}$, which is contradiction with $\fg_{\Xi-\Phi(\fk)-\Xi',-}\cap \fp_{\Phi(\fk)-\Xi',+}=0$. 
Therefore $Z':=2Z-[[\ta(X),Z],Z]$ is such that the component of $Z'$ in $\fg_{\Phi(\fk)-\Xi',0}\cap  \fp_{\Xi-\Phi(\fk)-\Xi',+}\cap \fg_{\Xi-\Xi',l}$ is non--trivial, and the lowest non--trivial component of $Z'$ is contained in $\fp_{\Phi(\fk)-\Xi',+}\cap  \fg_{\Xi-\Xi',0}\cap \fg_{\Xi-\Phi(\fk)-\Xi',j'}$ for some $j'>j$. Thus by repeating the construction we end up in the situation (a). 

Thus we have proven $\tfh\cap \fp_{\Xi-\Xi',+}\subset \fp_{\Phi(\fk)-\Xi',+}$ and $\tfh\cap \fp_{\Xi-\Phi(\fk),+}=0$ holds, because the construction from (a) can be used even in the cases $Z_{+0}=0$ or $Z_{+0}+Z_{0+}=0$ for $\Xi'=\emptyset$, and the first claim holds.

In the rest of the proof, we will assume that the parabolic geometry $(\ba\to M,\om)$ is non--flat and rely on the classification tables in \ref{ApC}. In particular, we see that either

\begin{itemize}
\item $\fp_{\Phi(\fk),+}\cap  \fg_{\Xi-\Phi(\fk),0}\subset \fg_{\Xi,1}$ holds or
\item $\fp_{\Phi(\fk),+}\cap  \fg_{\Xi-\Phi(\fk),0}\subset \fg_{\Xi,1}\oplus V_{\Xi,\delta'},$
which applies for

 $\frak{sp}(2n,\{\mathbb{R},\mathbb{C}\})$ with $\alpha_p\in \Phi(\fk)$ and $\delta'=2\alpha_p+\dots$, 

$\frak{so}(2n,\mathbb{C}),\frak{so}(n,n)$ with $\{\alpha_2,\alpha_n\}= \Phi(\fk)$  and $\delta'=\alpha_2+\dots+\alpha_n$,

$\frak{sl}(n+1,\{\mathbb{R},\mathbb{C}\})$ with $\{\alpha_p,\alpha_q\}\subset \Phi(\fk)$ and $\delta'=\alpha_p+\dots+\alpha_q$.
\end{itemize}

For any pair $V_{\Xi,\gamma}$ and $V_{\Xi,\gamma'}$ in $\fp_{\Phi(\fk),+}\cap  \fg_{\Xi-\Phi(k),0}\cap \fg_{\Xi,1}$, there always exists some $\Xi'\subset \Xi$ such that $V_{\Xi,\gamma}\subset \fg_{\Xi',0}$ and $V_{\Xi,\gamma'}\in \fp_{\Xi',+}$. This means that we can decompose $(Z_{+0})_j$ into $\fg_{\Xi',0}$ and $\fp_{\Xi',+}$ and repeat the construction from (b) in the proof of the first claim. In particular, we can choose $X_-\in \fg_{\Phi(\fk),-}\cap  \fg_{\Xi-\Phi(k),0}\cap \fg_{\Xi,-1}\cap \fg_{\Xi',0}$ and therefore the restriction of $\ad_{[\ta(X),Z]}$ to (complexification of) $V_{\Xi,\gamma'}$ is diagonalizable (up to components of $\Xi$--height higher than $1$).
Thus there is basis of $\tfh\cap \fp_{\Xi,+}$ such that $(Z_{+0})_j\in V_{\Xi,\delta}$ for some $\delta$ for each element $Z$ of the basis. Then the assumptions of the Proposition \ref{Aprop} are satisfied, which proves the second claim.

Clearly, there is a basis of $\tfh\cap \fp_{\Xi-\Theta(\fk),+}$ satisfying the first condition of the third claim.
 The same arguments as in the proof of the second claim can applied for $V_{\Xi-\Theta(\fk),\gamma}$ and
 $V_{\Xi-\Theta(\fk),\gamma'}$ in $\fp_{\Phi(\fk)-\Theta(\fk),+}\cap  \fg_{\Lambda(\fk),0}\cap \fg_{\Xi-\Theta(\fk),1}.$ Therefore there is a basis such that $(Z_{+0})_j\in V_{\Xi-\Theta(\fk),\delta}$ holds for some $\delta$ for every element $Z$ of the basis. Thus there is a basis satisfying the first two conditions for the cases when $\fp_{\Phi(\fk),+}\cap  \fg_{\Xi-\Phi(\fk),0}\subset \fg_{\Xi,1}$ holds.

There is a basis satisfying the first two conditions for the three exceptions with $V_{\Xi-\Theta(\fk),\delta'}$, too, because if $V_{\Xi-\Theta(\fk),\delta'}$ is in the same eigenspace w.r.t all $s\in \mathcal{J}(\fk)$ as $V_{\Xi-\Theta(\fk),\delta}$, then we see from the tables in the \ref{ApC} that there has to be $\delta=\alpha_1\in \Phi(\fk)$ in all cases and therefore the same arguments as above for $\Xi-\{\alpha_1\}-\Theta(\fk)$--heights show that we can change the basis.

Now the third condition of the third claim follows, because for the component of $Z$ in $\fg_{\Xi-\Theta(\fk),l}\cap \fp_{\Lambda(\fk),+}$, we can use the same argumentation as for $(Z_{0+})_{l}$. Therefore if the third condition is not satisfied, then we obtain contradiction as above. Thus we proved the third claim.

Assume $[\fm\cap \fg_{\Lambda(\fk),0}\cap \fp_{\Xi-\Theta(\fk),+},\fm\cap \fp_{\Lambda(\fk),+}]\neq 0$. It follows from definitions of $\Lambda(\fk)$ and $\Theta(\fk)$ that there are at least four simple restricted roots in $\Xi$, whose root spaces are not in $\fm$. However, it follows from the tables in the \ref{ApC} that there is only one possibility for $(\tG,\tP_\Xi)$ (up to complexification), and going trough the list of all possible eigenvalues of all possible $s\in \mathcal{J}(\fk)$, we see that $[\fm\cap \fg_{\Lambda(\fk),0}\cap \fp_{\Xi-\Theta(\fk),+},\fm\cap \fp_{\Lambda(\fk),+}]=0$ holds and the fourth claim follows. 
\end{proof}

The third claim of the previous Proposition has the following consequence.

\begin{cor*}
Suppose the parabolic geometry $(\ba\to M,\om)$ is non--flat. Let $(\ta,\ti)$ be extension of $(\tK,\tH)$ to $(\tG,\tP_\Xi)$ from the Theorem \ref{extension}. Let us denote by $\fm$ the intersection of all $1$--eigenspaces in $\fp_{\Xi,+}$ for all $s\in \mathcal{J}(\fk)$. If $\fm\cap \fp_{\Lambda(\fk),+}=0$, then $$\tH\subset \tilde Q_{\Lambda(\fk)}\cap \tP_\Xi.$$
\end{cor*}

The following Theorem is a generalization of the previous Corollary. In particular, it implies the Theorem \ref{main2}. 

\begin{thm*} \label{main}
Suppose the parabolic geometry $(\ba\to M,\om)$ is non--flat. Let $(\ta,\ti)$ be extension of $(\tK,\tH)$ to $(\tG,\tP_\Xi)$ from the Theorem \ref{extension}. Let us denote by $\fm$ the intersection of all $1$--eigenspaces in $\fp_{\Xi,+}$ for all $s\in \mathcal{J}(\fk)$. There is $p\in \exp(\fm\cap \fp_{\Lambda(\fk),+})$ such that 
$$\conj_p(\tH)\subset \tilde Q_{\Lambda(\fk)}\cap \tP_\Xi.$$
\end{thm*}
\begin{proof}
Since $\tfn\cap  \fp_{\Lambda(\fk),+}=0$, there is a Lie subalgebra $\tfn_0\subset \fg_{\Lambda(\fk),0}\cap \fp_{\Xi}$ isomorphic to $\tfn$ as Lie algebra and as $\tH$--module, which is given by the components of elements in $\tfn$ in $\fg_{\Lambda(\fk),0}\cap \fp_{\Xi}$. 
To prove the claim, it is enough to show that the isomorphism is given by $\Ad_p$ for some $p\in \exp(\fm\cap \fp_{\Lambda(\fk),+})$. Then the claim of the Theorem \ref{main} holds, because $\Ad_p: \tfn\to \tfn_0$ being $\tH$--module isomorphisms means that $p$ commutes with $\tH_0$.

Let us write $Z_0\mapsto Z_0+\phi(Z_0)$ for the inverse isomorphism $\tfn_0\to \tfn$ for $Z_0\in \tfn_0$ and $\phi(Z_0)\in \fp_{\Lambda(\fk),+}$. The map $\phi: \tfn_0\to \fp_{\Lambda(\fk),+}$ is linear, but only $\tH_0$--equivariant. We decompose $\phi(Z_0)$ into components $(Z_0)_\nu$ in submodules $V_{\Xi-\Theta(\fk),\nu}$ for all $\nu$ in $\fp_{\Lambda(\fk),+}$ and carry the proof step by step with respect to the module 
$V_{\Xi-\Theta(\fk),\beta}\subset \fp_{\Lambda(\fk),+}$ of the lowest $\Xi-\Theta(\fk)$--height that contains non--trivial component $(Z_0)_\beta$ of $\phi(Z_0)$ among all $Z_0\in \tfn_0$.

To show $\id+\phi=\Ad_p^{-1}$ for some $p\in \exp(\fm\cap \fp_{\Lambda(\fk),+})$, it is enough to prove the following two claims:

\begin{enumerate}
\item It holds $V_{\Xi-\Theta(\fk),\beta}\subset \fm.$
\item There is a fixed $Y_\beta\in V_{\Xi-\Theta(\fk),\beta}$ such that $[Y_\beta,Z_0]=(Z_0)_\beta$ holds for all $Z_0 \in \tfn_0$.
\end{enumerate}

If the claims (1) and (2) hold, then 
$$\Ad_{\exp(-Y_\beta)}(Z)=Z_0+(Z_0)_\beta+[-Y_\beta,Z_0]+\dots=Z_0+\dots$$
 holds for all $Z\in \tfn$, where $\dots$ are in the other components in the same or higher $\Xi-\Theta(\fk)$--height in $\fp_{\Lambda(\fk),+}$. This way we construct step by step $p\in \exp(\fm\cap \fp_{\Lambda(\fk),+})$ such that 
$\id+\phi=\Ad_p^{-1}.$

It is easy to prove the claim (2) if we assume that the claim (1) holds. 
It follows from the claim (4) of the Proposition \ref{2.31} that we can assume $Z_0\in \fg_{\Xi-\Theta(\fk),0}$, because the other components of $Z_0$ in $\fg_{\Lambda(\fk),0}$ act trivially on $\fm \cap \fg_{\Lambda(\fk),+}$ in our situation. There is $\alpha_i\in \Xi-\Theta(\fk)$ such that 
$$0\neq V_{\Xi-\Theta(\fk),\gamma}:=[V_{\Xi-\Theta(\fk),-\alpha_i},V_{\Xi-\Theta(\fk),\beta}]\subset \fp_{\Lambda(\fk),+}$$
for some $\gamma$ in $\fp_{\Lambda(\fk),+}$,
 because $V_{\Xi-\Theta(\fk),\beta}$ has $\Xi-\Theta(\fk)$--height at least $2$ by definition of $\Lambda(\fk)$. Then $\ad_{Z_0}$ preserves $V_{\Xi-\Theta(\fk),-\alpha_i}$ and $V_{\Xi-\Theta(\fk),\gamma}$, and 
$$\ad_{(Z_0)_\beta}\in V_{\Xi-\Theta(\fk),-\alpha_i}^*\otimes V_{\Xi-\Theta(\fk),\gamma}.$$
 In particular, there are elements $X,X'\in \tfk$ such that $\ta(X)$ has the component $X_-$ in $\fg/\fp_{\Xi}$ contained in $V_{\Xi-\Theta(\fk),-\alpha_i}$ and  $\ta(X')$ has the component  in $\fg/\fp_{\Xi}$ of the form $[Z_0,X_-]$. Thus $[\ta(X),Z]-\ta(X')\in \tfn$ and therefore $[(Z_0)_\beta,X_-]+[Z_0,X_\gamma]-X'_{\gamma}=0$ holds for the components 
$X_\gamma$ of $\ta(X)$ and $X'_{\gamma}$ of $\ta(X')$
in $V_{\Xi-\Theta(\fk),\gamma}$. The component $X_\gamma$ depends linearly on $X_-$ and the component $X'_{\gamma}$ depends linearly on $[Z_0,X_-]$, because $\gamma$ has lower $\Xi-\Theta(\fk)$--height than $\beta$. Let us denote this $\tH_0$--equivariant linear map $V_{\Xi-\Theta(\fk),-\alpha_i}\to V_{\Xi-\Theta(\fk),\gamma}$ by $\psi$, i.e, 
$$[(Z_0)_\beta,X_-]=\psi([Z_0,X_-])-[Z_0,\psi(X_-)]$$ 
holds for all $X_-\in V_{\Xi-\Theta(\fk),-\alpha_i}$. Therefore 
$$\ad_{(Z_0)_\beta}\in \ad_{Z_0}(V_{\Xi-\Theta(\fk),-\alpha_i}^*\otimes V_{\Xi-\Theta(\fk),\gamma}).$$ Since $V_{\Xi-\Theta(\fk),-\alpha_i}^*\otimes V_{\Xi-\Theta(\fk),\gamma}$ is completely reducible $\fg_{\Xi-\Theta(\fk),0}$--module and $V_{\Xi-\Theta(\fk),\beta}$ is identified with indecomposable $\fg_{\Xi-\Theta(\fk),0}$--submodule of $V_{\Xi-\Theta(\fk),-\alpha_i}^*\otimes V_{\Xi-\Theta(\fk),\gamma}$ by the adjoint action, there is $Y_\beta\in V_{\Xi-\Theta(\fk),\beta}$ such that $[Y_\beta,Z_0]=(Z_0)_\beta$ holds for all $Z_0\in \tfn_0$ and $Y_\beta$ is unique, because the map $\psi$ is independent on $Z_0$.

Therefore we need to prove the claim (1) to complete the proof. If $Z_0\in \fg_{\Xi,0}\oplus \fm$, then $V_{\Xi-\Theta(\fk),\beta}\subset \fm$ holds due to $\tH_0$--equivariancy. If $Z_0\in \fg_{\Lambda(\fk),0}\cap \fp_{\Xi,+}$, then we know from the claim (3) of the Proposition \ref{2.31} that we can assume $Z=Z_0+[Z_0,R]+\dots$ for $Z_0\in V_{\Xi-\Theta(\fk),\delta}$ for some $V_{\Xi-\Theta(\fk),\delta}\subset \fg_{\Lambda(\fk),0}\cap \fp_{\Xi,+}$ and $R\in \fm\cap \fp_{\Lambda(\fk),+}$ and that $V_{\Xi-\Theta(\fk),\beta}\subset [\fm\cap \fp_{\Lambda(\fk),+},V_{\Xi-\Theta(\fk),\delta}]$. 

Consider $X\in \tfk$ in a single eigenspace for each $s\in \mathcal{J}(\fk)$ such that the projection of $\ta(X)$ into $\fg/\fp_{\Xi}$ is contained in $V_{\Xi-\Theta(\fk),-\delta}$. That is $\ta(X)=X_-+X_0+X_+$ for some $X_-\in V_{\Xi-\Theta(\fk),-\delta}$, $X_0\in \fg_{\Lambda(\fk),0}\cap \fp_{\Xi,+}$ and $X_+\in \fp_{\Lambda(\fk),+}$. Then \begin{align*}
[\ta(X),Z]=&[X_-,Z_0]+[X_0,Z_0]
+[X_-+X_0,[Z_0,R]]\\
+&[X_-,(Z_0)_\beta]+[X_+,Z_0+[Z_0,R]]+\dots
\end{align*}
 holds for $\dots$ in components of $\fp_{\Lambda(\fk),+}$ of higher $\Xi-\Theta(\fk)$--height than $[X_-,(Z_0)_\beta]\in [V_{\Xi-\Theta(\fk),-\delta},V_{\Xi-\Theta(\fk),\beta}]$. The first three summands of $[\ta(X),Z]$ are contained in $\fg_{\Lambda(\fk),0}\cap \fp_{\Xi}$ and 
$$[X_+,Z_0+[Z_0,R]]=[Z_0,-X_++[X_+,R]],$$
 because $[[Z_0,X_+],R]=0$ follows from the claim (4) of Proposition \ref{2.31}. If $[Z_0,-X_++[X_+,R]]$ 
has trivial component in $[V_{\Xi-\Theta(\fk),-\delta},V_{\Xi-\Theta(\fk),\beta}]$, then $[X_-,(Z_0)_\beta]=0$ holds, because $[V_{\Xi-\Theta(\fk),-\delta},V_{\Xi-\Theta(\fk),\beta}]$ has lower $\Xi-\Theta(\fk)$--height than $V_{\Xi-\Theta(\fk),\beta}$. But this is contradiction with $(Z_0)_\beta\neq 0$, because there is always $X_-\in V_{\Xi-\Theta(\fk),-\delta}$ such that $[X_-,(Z_0)_\beta]\neq 0$. So $[Z_0,-X_++[X_+,Y]]$ has the non--trivial component in $[V_{\Xi-\Theta(\fk),-\delta},V_{\Xi-\Theta(\fk),\beta}]$, that is 
$$[V_{\Xi-\Theta(\fk),-\delta},V_{\Xi-\Theta(\fk),\beta}]=[V_{\Xi-\Theta(\fk),\delta},V_{\Xi-\Theta(\fk),\eta}]$$ for some $V_{\Xi-\Theta(\fk),\eta}\subset \fp_{\Lambda(\fk),+}.$ Thus
$$V_{\Xi-\Theta(\fk),\beta}=[V_{\Xi-\Theta(\fk),\delta},[V_{\Xi-\Theta(\fk),\delta},V_{\Xi-\Theta(\fk),\eta}]].$$ 
We prove the claim $V_{\Xi-\Theta(\fk),\beta}\subset \fm$  case by case by the discussion according to the possible cases from the tables in the \ref{ApC} that can satisfy $V_{\Xi-\Theta(\fk),\beta}=[V_{\Xi-\Theta(\fk),\delta},[V_{\Xi-\Theta(\fk),\delta},V_{\Xi-\Theta(\fk),\eta}]].$ Clearly, there has to be a restricted root $\alpha_i\in I_\mu$ such that the highest root of $\fg$ has $\{\alpha_i\}$--height at least $2$. Thus there are the following possibilities:

\begin{enumerate}
\item $\fg=\frak{sp}(2n,\{\mathbb{R},\mathbb{C}\}), \{1,2\}\subset \Xi, V_{\Xi-\Theta(\fk),\alpha_1+2\alpha_2+\dots}\subset \fm\cap \fp_{\Lambda(\fk),+}$ 
\item $\fg=\frak{sp}(2n,\{\mathbb{R},\mathbb{C}\}), \{1,2,p\}=\Xi,n>p, V_{\Xi-\Theta(\fk),\alpha_2+\dots+\alpha_p}\subset \fm\cap \fp_{\Lambda(\fk),+}$ 
\item $\fg=\frak{so}(q,n-q),\frak{so}(n,\mathbb{C}), \{1,2\}\subset \Xi, V_{\Xi-\Theta(\fk),\alpha_1+\alpha_2}\subset \fm\cap \fp_{\Lambda(\fk),+}$
\end{enumerate}

The case (1) follows trivially, because of dimensional reasons. Indeed, we can conjugate by arbitrary $\exp(Y)\in \exp(V_{\Xi-\Theta(\fk),\alpha_1+2\alpha_2+\dots})$ in the last step and prove the Theorem \ref{main} in this case, because $Z_0\in V_{\Xi,\alpha_i}$ is the only component of any $Z_0\in \fn_0$ that acts non--trivially on $V_{\Xi-\Theta(\fk),\alpha_1+2\alpha_2+\dots}$.

The cases (2) and (3) are similar as $\fg_{\Xi,0}$--modules. Only in the case (2) all the modules are in addition in the tensor product with standard representation of one of the simple components of $\fg_{\Xi,0}$. Thus it is enough to prove $V_{\Xi-\Theta(\fk),\beta}\subset \fm$ in the case (3). In this case, $[X_+,Z_0]+[X_-,(Z_0)_\beta]=(Z_0)_\beta\cdot X_-^T+X_+\cdot Z_0=0$ has to hold due to dimensional reasons, where we identify $V_{\Xi-\Theta(\fk),\alpha_1+\alpha_2}$ with $V_{\Xi-\Theta(\fk),\alpha_2}$ and we identify $V_{\Xi-\Theta(\fk),\alpha_1}$ and $V_{\Xi-\Theta(\fk),\alpha_1+2\alpha_2+\dots}$ with $\mathbb{R}$ or $\mathbb{C}$. If $X_-^T$ is not a multiple $Z_0$, then $X_+=0$ and thus $(Z_0)_\beta\cdot X_-^T=0$ holds, because $X_+$ depends linearly on $X_-^T$. But this is contradiction with $(Z_0)_\beta\neq 0$ and thus $V_{\Xi-\Theta(\fk),\beta}\subset \fm$ holds in this case.
\end{proof}

\section{Almost invariant Weyl structures}

Let us return to the reductions introduced in the Section \ref{dva}.
We define here distinguished classes of Weyl structures compatible with the reductions $(\ba_{\Xi'}\to M,\om^{\Xi'})$ from the Theorem \ref{5.1}.   We discuss them in general here and we show the role of generalized symmetries in the subsequent sections. We use the concept of Weyl structures as introduced in \cite[Section 5.]{parabook}.

\begin{def*}
 Let $\ba_0 \simeq \ba/\exp{(\fp_{\Xi,+})}$  be the underlying $G_0$--bundle.
We call $G_0$--equivariant sections $\si: \ba_0\to \ba_{\Xi'}\subset \ba$ \emph{almost $\Xi'$--invariant Weyl structures}. If $\ba_0=\ba_{\Xi'}$, then we call the Weyl structure $\si=\iota: \ba_0\to \ba$ \emph{invariant}.
\end{def*}

It is easy to describe the situation $\ba_0=\ba_{\Xi'}$.  The Proposition \ref{5.3} implies $T^{\Xi',-}M=TM$, and the soldering form and the connection from the Proposition \ref{5.3} correspond to the soldering form and the connection given by the invariant Weyl structure. 

Let us discuss the existence of general almost $\Xi'$--invariant Weyl structures and investigate their properties.

\begin{prop*}\label{lamweyl}
Almost $\Xi'$--invariant Weyl structures always exist on geometries satisfying assumptions of the Theorem \ref{5.1}, and form an affine space over sections of $gr(V^*M):=\ba_0\times_{\Ad_{G_0}} (\fq_{\Xi'}\cap \fp_{\Xi,+})$.

Moreover, the pull--back $\si^*\om$ coincides on $\fq_{\Xi'}$ with the pull--back $\si^*\om^{\Xi'}$  for any almost $\Xi'$--invariant Weyl structure $\si$.
\end{prop*}
\begin{proof}
The proof of the first claim follows in the same way as the proof of the existence of global Weyl structures in \cite[Section 5.1.1]{parabook}. The only difference is that we consider $G_0\exp(\fq_{\Xi'}\cap \fp_{\Xi,+})$ instead of $G_0\exp(\fp_+)$. The second claim  follows from the construction.
\end{proof}

The above Proposition particularly clarifies the name almost $\Xi'$--invariant Weyl structure $\si$, because the pullback of $\si$ by arbitrary automorphism of the parabolic geometry $(\ba_{\Xi'}\to M,\om^{\Xi'})$ equals to $\si \exp(\Upsilon)$ for suitable $\U$ satisfying $Im(\Upsilon)\subset \fq_{\Xi'}\cap \fp_{\Xi,+}$.

Let $\si$ be an arbitrary fixed almost $\Xi'$--invariant Weyl structure. We can decompose the pullback $\si^*\om$ into $G_0$--invariant parts which can be interpreted as follows:

\begin{prop*} \label{5.5}
In the setting as above:

\begin{enumerate}
\item The part of $\si^*\om$ valued in $\fg_{\Xi',-}$ provides isomorphism $$T^{\Xi',-}M \simeq \ba_0\times_{G_0}\fg_{\Xi',-}.$$
Moreover, the part of $\si^*\om$ valued in arbitrary $\fg_{\Xi',0}$--submodule $\fd$ of $\fg_{\Xi',-}$ provides isomorphism of $\ba_0\times_{G_0}\fd$ with the $\Aut(\ba^{\Xi'},\om^{\Xi'})$--invariant subbundle $\ba^{\Xi'}\times_{Q_{\Xi'}\cap P} \fd$ of $T^{\Xi',-}M$.
\item The part of $\si^*\om$ valued in $\fg_{\Xi',0}$ provides a unique (and thus $\Aut(\ba^{\Xi'},\om^{\Xi'})$--invariant) affine connection on each $\ba_0\times_{G_0}\fd$ from (1).
\item The part of $\si^*\om$ valued in $\fg_{\Xi',0}\cap \fg_{\Xi,-}$ provides isomorphism of $gr(VM)$ with $VM$.
\item The part of $\si^*\om$ valued in $\fg_{\Xi',0}\cap \fg_{\Xi,0}$ provides an affine connection on $VM$.
\item The part of $\si^*\om$ valued in $\fg_{\Xi',0}\cap \fp_{\Xi,+}$ provides isomorphism of $gr(V^*M)$ and $V^*M$.
\item The part of $\si^*\om$ valued in $\fp_{\Xi',+}$ provides isomorphism $$(T^{\Xi',-}M)^* \simeq \ba_0\times_{G_0}\fp_{\Xi',+}.$$
Moreover, the part of $\si^*\om$ valued in arbitrary $\fg_{\Xi',0}$--submodule $\fd^*$ of $\fp_{\Xi',+}$, which is dual (via Killing form of $\fg$) to $\fg_{\Xi',0}$--submodule $\fd$ of $\fg_{\Xi',-}$, provides isomorphism of $(\ba_0\times_{G_0}\fd)^*$ with $\ba^{\Xi'}\times_{Q_{\Xi'}\cap P} \fd^*$.
\end{enumerate}
\end{prop*}
\begin{proof}
The claim follows directly from the Propositions \ref{5.3} and \ref{lamweyl} and from the definition and properties of Weyl structures, see \cite[Section 5.1.]{parabook}.
\end{proof}
There is the following particular consequence of the above statement.
\begin{cor*}
If we apply the Proposition \ref{5.5} on the reduction $(\ba_{\Lambda}\to M,\om^{\Lambda})$ from the Theorem \ref{main2}, we get that all $\Aut(\ba^{\Lambda},\om^{\Lambda})$--invariant objects are $\A$--invariant.
\end{cor*}
In the next section, we focus on the integrability of distributions mentioned above.

\section{Generalized symmetries on correspondence and twistor spaces}

Let us discuss here possible correspondence and twistor spaces in the sense of  \cite{Cap-correspondence} and the role of generalized symmetries on them.

Let us first consider arbitrary correspondence space of our parabolic geometry $(\ba\to M,\om)$. Clearly, the correspondence space does not have to satisfy our assumptions anymore. Indeed, it does not have to be homogeneous. However, the generalized symmetries of the parabolic geometry downstairs are lifted to generalized symmetries of appropriate type on the correspondence space. Moreover, there can be generalized symmetries, that do not cover generalized symmetries downstairs. It is easy to characterize the lifted symmetries.

 \begin{prop*}\label{cor}
A generalized symmetry at $x$ on a correspondence space to $(\ba\to M,\om)$ is a lift of an underlying generalized symmetry of the parabolic geometry $(\ba\to M,\om)$ if and only if the vertical bundle $\mathcal{V}_xM$ is contained inside the $1$--eigenspace of the generalized symmetry.
\end{prop*}
\begin{proof}
Clearly, the condition that $\mathcal{V}_xM$ is inside the $1$--eigenspace of the generalized symmetry is necessary. Since $Ker(\Ad_{G_0}|_{\fg_0})=Z(G_0)$, it is also sufficient.
\end{proof}

Now, let us focus on (local) twistor spaces of the geometry $(\ba \to M, \om)$.
Let us fix $ s\in Z(G_0)$. Then we can characterize, when $s$--symmetry is (locally) a lift of a (local) generalized symmetry from a suitable (local) twistor space. 

\begin{prop*}\label{4.2}
Suppose there is $s$--symmetry in $\A$ and denote by $\Psi(1)$ the set of all simple restricted roots $\alpha_i\in \Psi$ such that $\fg_{\alpha_i}$ is in $1$--eigenspace of $\Ad_s$. Then the following statements hold:
\begin{enumerate}
\item The parabolic geometry $(\ba\to M,\om)$ is locally equivalent to an open subset of a correspondence space to a parabolic geometry of type $(G,P_{\Xi-\Psi(1)})$. 
\item The parabolic geometry of type $(G,P_{\Xi-\Psi(1)})$ from (1) is locally homogeneous with local generalized symmetries covered by $s$--symmetries in $\A$.
\item The preimage $\fl$ of $\fp_{\Xi-\Psi(1)}$ in $\fk$ is a Lie subalgebra of $\fk$ for any extension $(\alpha,i)$ of $(K,H)$ to $(G,P)$ giving $(\ba\to M,\om)$ at any $u_0$.
\end{enumerate}
\end{prop*}
\begin{proof}
By the definition of $\Psi$, the harmonic curvature vanishes on insertions of elements of $\fp_{\Xi-\Psi(1)}$. Then the existence of the twistor space follows  by the results in \cite[Sections  1.5.13,14]{parabook} or \cite{Cap-correspondence}, 
and the existence of local generalized symmetries follows from the Proposition \ref{cor}. Moreover, since the whole curvature vanishes on the insertions of entries of $\fp_{\Xi-\Psi(1)}$, the last claim follows.
\end{proof}

Let us remark that although the parabolic geometry on the twistor space is normal, it does not have to be regular. However, it follows from the classification in the  \ref{ApC} that this behavior applies only for the following cases:

\begin{lem*}
The parabolic geometries of type $(G,P_{\Xi-\Psi(1)})$ from the Proposition \ref{4.2} are not regular in the following cases:

\begin{itemize}
\item $\frak{sl}(n+1,\{\mathbb{R,C}\})$ with $\Xi=\{1,2,p,q\}$ and $\Psi(1)=\{\alpha_1\}$,
\item $\frak{sp}(2n,\{\mathbb{R,C}\})$ with $\Xi=\{1,2,p\}$, where $p<n$, and $\Psi(1)=\{\alpha_1\}$.
\end{itemize}
\end{lem*}

Let us now show, when it is possible to construct the twistor space globally for the parabolic geometry $(\ba\to M,\om)$ given by extension $(\alpha,i)$ of $(K,H)$ to $(G,P)$ at $u_0$. Firstly, we show a sufficient condition for the global existence of the leaf space, i.e., we show when there is a closed Lie subgroup $L$ with the Lie algebra $\fl$ (from claim (3) of Proposition \ref{4.2}) containing $H$.

\begin{lem*}\label{5.4}
Suppose there is $s'\in \iota(H)\cap Z(G_0)$ such that $1$--eigenspace $\fl(1)$ of $\Ad_{s'}$ in $\fl$ coincides with $1$--eigenspace of $\Ad_{s}$ and $\Ad_{s'}|_{\fg_{\Xi,-k}}\neq \id$ on $|k|$--graded geometry. Then the subgroup $L$ of $K$, which is generated by the group $L(1)$ containing the elements $l$ of the commutator $Z_{K}(s')$ of $s'$ in $K$ such that $\Ad_l\in \tP_{\Xi-\Psi(1),+}$ together with the group $\exp(\fh\cap \fp_{\Xi-\Psi(1),+})$, is closed Lie subgroup of $K$ with Lie algebra $\fl$ containing $H$.
\end{lem*}
\begin{proof}
Let us point out that the condition $\Ad_l\in P_{\Xi-\Psi(1),+}$ means that $\Ad_l$ preserves the filtration of $\fk$ induced by $\Xi-\Psi(1)$--heights and preserves (if exists) the tensor field providing the reduction to $G_0.$

Since $\exp(\fh\cap \fp_{\Xi-\Psi(1),+})$ is stable under conjugation by elements of $L(1)$ by definition, it is enough to show that the closed Lie subgroup $L(1)$ of $K$ has the Lie algebra $\fl(1)$. By our assumption, $\fl(1)$ is contained in the Lie algebra of $L(1)$.

For the regular geometries, we show it from the regularity. If $Z$ from the Lie algebra of $L(1)$ has a non--trivial projection to $\fg_{\Xi,i}$, then $i\neq -k$ by assumption. If $i<0$, then there is $X\in \alpha(\fk)$ such that $[Z,X]$ has component of lower $\Xi$--height than $X$,
 which is contradiction with the definition of $L(1)$.

In the two  non--regular cases from above, the claim follows by the same arguments, because either $X$ or $Z$ are in $\fg_{\Xi-\Psi,0}$, where the curvature does not have entries. Thus it remains to show that $H\subset L$, but we know from \cite[Theorem 4.1]{GZ2} that $H$ decomposes into $Z_{H}(s')$ and $\exp(\fh\cap \fp_{\Xi-\Psi(1),+})$.
\end{proof}

Let us remark that if there is no such $s'\in \iota(H)\cap Z(G_0)$ as in the previous Lemma, then the centralizer $Z_{K}(s)$ can have a non--trivial projection to $\fg_{\Xi,-k}$. Since such elements in the centralizer can preserve the filtration, the subgroup $L$ constructed from $Z_{K}(s)$ does not have to be closed in general situation.

Further, we show that we can generally construct  only a parabolic geometry of type $(\tG,\tP_{\Xi-\Psi(1)})$ on the twistor space.

\begin{prop*}\label{4.4}
Suppose $L$ is a closed Lie subgroup of $K$ with Lie algebra $\fl$ containing $H$.
 Then the pair $(\alpha,\Ad)$ defines an extension of $(K,L)$ to $(\tG,\tP_{\Xi-\Psi(1)})$, and the parabolic geometry $(\ba\to M,\om)$ covers an open subset in the correspondence space to the parabolic geometry
$$(K\times_{\Ad(L)} \tP_{\Xi-\Psi(1)}\to K/L,\om_\alpha)$$ given by this extension. In particular, there is a natural inclusion of the Lie group $Ker(\Ad|_{L})\backslash K$ into the group of automorphisms of $(K\times_{\Ad(L)} \tP_{\Xi-\Psi(1)}\to K/L,\om_\alpha)$.
\end{prop*}
\begin{proof}
The first claim is clear from the construction of $L$. Then the kernel of the natural map $K\to K\times_{\Ad(L)} \tP_{\Xi-\Psi(1)}$ equals to $Ker(\Ad|_{L})$, which is a normal subgroup of $Ker(\Ad)$ and thus a normal subgroup of $K$. Therefore the kernel of the quotient map $M=K/H\to K\times_{\Ad(L)} \tP_{\Xi-\Psi(1)}/\tP_{\Xi}$ equals to $Ker(\Ad_H)$, the kernel of $\ba=K\times_H P\to K\times_{\Ad(L)} \tP_{\Xi-\Psi(1)}$ equals to $Ker(\Ad_P)$ and the rest follows.
\end{proof}

The intersection of the sets $\Theta$ and $\Psi$ is contained in $\Psi(1)$ for any $s\in \mathcal{J}$. We can generalize the previous Proposition in the following way.

\begin{thm*} \label{4.5}
Suppose there is $s\in \mathcal{J}$ such that $\fg_{\Xi,-k}$ is not contained in $1$--eigenspace of $\Ad_s$ on $|k|$--graded geometry. Let $L_{\Theta\cap \Psi}$ be the closed Lie subgroup of $\A$ constructed in the Lemma \ref{5.4} for choices $K=\A$ and $\Psi(1)=\Theta\cap \Psi$. The pair $(\alpha,\Ad)$ defines an extension of $(\A,L_{\Theta\cap \Psi})$ to $(\tG,\tP_{\Xi-(\Theta\cap \Psi)})$, and the parabolic geometry $(\ba\to M,\om)$ covers an open subset in the correspondence space to the parabolic geometry
$$(\A\times_{L_{\Theta\cap \Psi}} \tP_{\Xi-(\Theta\cap \Psi)}\to K/L_{\Theta\cap \Psi},\om_\alpha)$$ given by the extension $(\alpha,\Ad)$.

Moreover, the generalized symmetries in $\A$ cover generalized symmetries of $(\A\times_{L_{\Theta\cap \Psi}} \tP_{\Xi-(\Theta\cap \Psi)}\to K/L_{\Theta\cap \Psi},\om_\alpha)$, i.e., if $\Lambda\subset \Xi-\Phi-\Theta$, then $\Lambda$ is admissible choice in the Theorem \ref{main2} applied on the parabolic geometry $(\A\times_{L_{\Theta\cap \Psi}} \tP_{\Xi-(\Theta\cap \Psi)}\to K/L_{\Theta\cap \Psi},\om_\alpha)$. 
\end{thm*}
\begin{proof}
The first part of the Theorem is a direct consequence of the Proposition \ref{4.4}, because we are just taking the correspondence space to $P_{\Xi-(\Theta(\fk)\cap \Psi)}$ for the choice $K=\A$. The rest follows from the Proposition \ref{cor}.
\end{proof}

Let us finally investigate the relation between the set $\Psi$ and the decomposition $TM=T^{\Lambda,-}M\oplus VM$ given by the reduction from the Theorem \ref{main2}.
The corollaries follows directly from the Theorem  \ref{main2} and the Proposition \ref{5.5}.

Firstly, we can decompose $T^{\Lambda,-}M$ into subbundles corresponding to $\fg_{\Lambda,0}$--sub--modules of $\fg_{\Lambda,-}$. In particular, we consider subbundles corresponding to $\fg_{\Xi-\Phi}\cap \fg_{\Lambda,0}$--submodules.

\begin{cor*}
Suppose $\fd$ is a $\fg_{\Lambda,0}$--subbmodule of $\fg_{\Xi-\Psi,0}\cap \fg_{\Lambda,-}$. Then the subbundle $\A\times_{\A_{x_0}} \fd$ is integrable. Let $\mathcal{D}$ be the leaf of $\A\times_{\A_{x_0}} \fd$. Then $\mathcal{D}$  is initial submanifold of $M$ and carries a flat homogeneous geometry of type $(\exp{\fd}\rtimes( \tG_{\Lambda,0}\cap \tP_{\Xi}),\tG_{\Lambda,0}\cap \tP_{\Xi})$.
\end{cor*}
Clearly,  $\mathcal{D}$ does not have to be embedded and its stabilizer in $\A$ does not have to be closed in $\A$ or form reductive pair with $\A_{x_0}$.

Let us define the following complementary subbundles of $TM$, which together generate the bundle $VM=\ba_{\Lambda}\times_{(Q_{\Lambda}\cap P)} \fg_{\Lambda,0}/(\fg_{\Lambda,0}\cap \fp_{\Xi})$:
\begin{itemize}
\item $V^{+}M:= \ba_{\Lambda}\times_{(Q_{\Lambda}\cap P)} \fg_{\Lambda\cup (\Xi\cap \Psi),0}/(\fg_{\Lambda\cup (\Xi\cap \Psi),0}\cap \fp_{\Xi})$,
\item $V^{-}M:= \ba_{\Lambda}\times_{(Q_{\Lambda}\cap P)} \fg_{\Lambda\cup (\Theta- \Psi),0}/(\fg_{\Lambda\cup (\Theta-  \Psi),0}\cap \fp_{\Xi})$.
\end{itemize}

\begin{cor*}
The subbundle $V^{-}M$ is integrable. Let $\mathcal{V}^{-}$ be the leaf of $V^{-}M$. Then $\mathcal{V}^{-}$  is an initial submanifold of $M$, and any almost $\Lambda$--invariant Weyl structure restricted to $\mathcal{V}^{-}$ prolongs (as $\Ad_{\tG_{\Xi,0}}|_{\fg_{\Lambda\cup (\Theta- \Psi),0}/(\fg_{\Lambda\cup (\Theta-  \Psi),0}\cap \fp_{\Xi})}$--structure) to a flat homogeneous Cartan geometry of type $(\tG_{\Lambda\cup (\Xi\cap \Psi)),0},\tG_{\Lambda\cup (\Xi\cap \Psi),0}\cap \tP_{\Xi})$.
\end{cor*}

Let us point out that the pair $(\tG_{\Lambda\cup (\Xi\cap \Psi)),0},\tG_{\Lambda\cup (\Xi\cap \Psi),0}\cap \tP_{\Xi})$ is not effective and the effective quotient is general parabolic geometry.

Again, $\mathcal{V}^{-}$ does not have to be embedded and its stabilizer in $\A$ does not have to be closed in $\A$.

If we go through the tables in the  \ref{ApC} to see, where can the harmonic curvature have its entries, then we obtain the following claim.

\begin{cor*}
The subbundle $V^{+}M$ is integrable. Let $\mathcal{V}^{+}$ be the leaf of $V^{+}M$. Then $\mathcal{V}^{+}$  is an initial submanifold of $M$ and each almost $\Lambda$--invariant Weyl structure restricted to $\mathcal{V}^{+}$ prolongs (as $\Ad_{\tG_{\Xi,0}}|_{ \fg_{\Lambda\cup (\Xi\cap \Psi),0}/(\fg_{\Lambda\cup (\Xi\cap \Psi),0}\cap \fp_{\Xi})}$--structure) to a flat homogeneous Cartan geometry of type $(\tG_{\Lambda\cup (\Theta-\Psi),0},\tG_{\Lambda\cup (\Theta-  \Psi),0}\cap \tP_{\Xi})$.
\end{cor*}

Again the pair $(\tG_{\Lambda\cup (\Theta-\Psi),0},\tG_{\Lambda\cup (\Theta-  \Psi),0}\cap \tP_{\Xi})$ is not effective and the effective quotient is general parabolic geometry.

Again, $\mathcal{V}^{+}$ does not have to be embedded and its stabilizer does not have to be closed in $\A$. Moreover, even if the stabilizer is closed, the curvature generally takes values in $V^{+}M$, and thus there is no induced parabolic geometry on the twistor space.

\section{Parabolic geometries carrying invariant Weyl structures in generic situations}

There are many geometries for which there is an invariant Weyl structure generically. In particular, it turns out that the most studied types of parabolic geometries are among them.

\begin{thm*}\label{7.1}
Assume there is a non--trivial $s$--symmetry in $\A$.
The set $\Lambda$ can be chosen as $\Xi$ in the Theorem \ref{main2}, if one of the following conditions holds for the parabolic geometry $(\ba\to M,\om)$:
\begin{itemize}
\item $\fg_{\Xi,-1}$ is indecomposable $\fg_{\Xi,0}$--module and $\kappa_H\neq 0$, or
\item the pair $(G,P)$ and the (subset of the set of) non--trivial components of $\kappa_H$ correspond to one of the entries in the table.

\begin{center}
\begin{tabular}{|c|c|c|}
\hline
$\fg$ & $\Xi$ & $\kappa_H$ \\
\hline
$\frak{sl}(n+1,\{\R,\C\})$ & $\{1,n\}$ &  $(\alpha_1,\alpha_n)$  \\
$\frak{sl}(n+1,\C)$ & $\{1,n\}$ &  $(\alpha_1,\alpha_2),(\alpha_{n},\alpha_{n-1})$  \\
$\frak{sl}(n+1,\C)$ & $\{1,n\}$ &  $(\alpha_1,\alpha_{n'}),(\alpha_{n},\alpha_{1'})$  \\
$\frak{sl}(n+1,\C)$ & $\{1,n\}$ &  $(\alpha_1,\alpha_{n'}),(\alpha_{1},\alpha_{2})$  \\
$\frak{sl}(n+1,\C)$ & $\{1,n\}$ &  $(\alpha_{n},\alpha_{n-1}),(\alpha_{n},\alpha_{1'})$  \\
\hline
$\frak{so}(3,4)$&$\{ 1,3\}$& $(\alpha_3,\alpha_2)$ \\
\hline
$\frak{sl}(n+1,\C)$ & $\{1,2\}$ &  $(\alpha_1,\alpha_2),(\alpha_{2},\alpha_{1})$  \\
$\frak{sl}(n+1,\C)$ & $\{1,2\}$ &  $(\alpha_1,\alpha_2),(\alpha_{1},\alpha_{2'})$  \\
$\frak{sl}(n+1,\C)$ & $\{1,2\}$ &  $(\alpha_1,\alpha_2),(\alpha_{2'},\alpha_{1'})$  \\
$\frak{sl}(n+1,\C)$ & $\{1,2\}$ &  $(\alpha_1,\alpha_2),(\alpha_{1'},\alpha_{2'})$  \\
$\frak{sl}(n+1,\C)$ & $\{1,2\}$ &  $(\alpha_2,\alpha_1),(\alpha_{1},\alpha_{2'})$  \\
$\frak{sl}(n+1,\C)$ & $\{1,2\}$ &  $(\alpha_2,\alpha_1),(\alpha_{1'},\alpha_{2'})$  \\
\hline
$\frak{sl}(n+1,\C)$ & $\{1,p\}$ &  $(\alpha_1,\alpha_2),(\alpha_{1},\alpha_{p'})$  \\
\hline
$\frak{sl}(n+1,\C)$ & $\{2,n-1\}$ &  $(\alpha_2,\alpha_1),(\alpha_{n-1},\alpha_{n})$  \\
\hline
$\frak{sl}(n+1,\C)$ & $\{1,2,n\}$ &  $(\alpha_1,\alpha_2),(\alpha_2,\alpha_1),(\alpha_{1},\alpha_{n})$  \\
\hline
$\frak{sp}(4,\C)$ & $\{1,2\}$ &  $(\alpha_1,\alpha_2),(\alpha_{1'},\alpha_{2'})$  \\
$\frak{sp}(4,\C)$ & $\{1,2\}$ &  $(\alpha_1,\alpha_2),(\alpha_{1},\alpha_{2'})$  \\
$\frak{sp}(4,\C)$ & $\{1,2\}$ &  $(\alpha_{1'},\alpha_{2'}),(\alpha_{1},\alpha_{2'})$  \\
$\frak{sp}(2n,\C)$ & $\{1,2\}$ &  $(\alpha_1,\alpha_2),(\alpha_{2},\alpha_{1})$  \\
\hline
\end{tabular}
\end{center}
\end{itemize}
In such case, the holonomy reduction from the Theorem \ref{main2} is a Cartan geometry $(\ba_0\to M,\om^\Xi)$ of type $(Q_\Xi,G_0)$.

Moreover, the following claims hold for such parabolic geometries:
\begin{enumerate}
\item For each $s\in Z(G_0)$, there is at most one $s$--symmetry at each point of $M$. 
\item There always exists a (local) $s$--symmetry of $(\ba\to M,\om)$ at $x_0$ (and thus at each point) for some $s\neq e$ of a finite order, i.e., there is $n \in \mathbb{N}$ such that $s^{n}=e$.
\item There is a smooth manifold $N$ and surjective submersion $\pi: M\to N$ such that $Ker(T\pi)$ consists of the $1$--eigenspaces of the (local) $s$--symmetries of finite order from (2). The $s$--symmetries descend to $N$ and determine a structure of a (locally) $\mathbb{Z}_n$--symmetric space on $N$ (see \cite{kowalski}).
\item There always exist invariant Weyl structures, and the corresponding invariant Weyl connections descend to $N$ and correspond to the canonical connection of the $\mathbb{Z}_n$--symmetric space $N$.
\end{enumerate}
\end{thm*}
\begin{proof}
Clearly, if $\fg_{\Xi,-1}$ is indecomposable $\fg_{\Xi,0}$--module, then $\Phi=\emptyset$ and if $s\neq e$, then $\Theta=\emptyset$. By looking in the tables in  \ref{ApC}, we get the remaining cases, when $\Phi=\emptyset$ and $\Theta=\emptyset$ generically holds. 

Since $\Phi=\emptyset$, the claim (1) holds. It follows from the classification in the  \ref{ApC} that either the claim (2) holds for $s$, 
or that $\fq_\Xi$ and $\wedge^2\fp_{\Xi,+}\otimes \fq_\Xi$ decomposes into indecomposable $\fg_{\Xi,0}$--modules according to eigenvalues of $s$ and thus there is $h\in Z(G_0)$ of finite order satisfying assumptions of \cite[Theorem 4.1.]{GZ2} and the claim (2) follows. Then by uniqueness of (local) $s$--symmetries and homogeneity, the $1$--eigenspaces form a subbundle of $TM$, which is invariant and integrable, because it corresponds to $1$--eigenspace of $\Ad_s$ in $\fg_{\Xi,-}$. If we denote $L$ the subgroup of $\A$ fixing the leaf of this integrable distribution trough $x_0$, then $L$ corresponds to $l\in \A$ such that $\Ad_l$ commutes with $\Ad_s$ and thus it is a closed subgroup of $\A$. Thus $N=\A/L$ and the claim (3) holds. Finally, the claim (4) is a consequence of the invariance of the splitting of $\fg_{\Xi,-}$ given by the invariant Weyl structure, which clearly has to exist.
\end{proof}

Let us point out that the list in the previous Theorem is complete and there are no other types of parabolic geometries for which $\Lambda$ can be chosen as $\Xi$ generically. 

Let us now formulate the results for the distinguished classes of parabolic geometries in detail.

\begin{prop*}
Let $(G,P)$ be type of a $|1|$--graded parabolic geometry. Assume $\kappa\neq 0$ and assume there is a non--trivial $s$--symmetry of $(\ba\to M,\om)$ at $x_0$. Then $M=N$ is either (locally) symmetric or (locally) $\mathbb{Z}_3$--symmetric space, and the invariant Weyl structure is unique.
\end{prop*}
\begin{proof}
All claims follow directly from the Theorem \ref{7.1} and the classification in the \ref{ApC}.
\end{proof}

We remark that we investigated $s$--symmetries of $|1|$--graded parabolic geometries order $2$ in detail in \cite{LZ1,disertace}. The theory developed in this articles does not provide any new results for homogeneous $|1|$--graded geometries with these symmetries except the fact that only existence of $s$--symmetries of the parabolic geometry given by extension $(\ta,\ti)$ from Theorem \ref{extension} is sufficient.

\begin{prop*} \label{7.3}
Let $(G,P)$ be type of a parabolic contact geometry. Assume $\kappa\neq 0$ and assume there is a non--trivial $s$--symmetry of $(\ba\to M,\om)$ at $x_0$. If $\Lambda=\Xi$, then one of the following facts applies:
\begin{itemize} 
\item $ker(T_x\pi)\cong \fg_{-2}$ for each $x$ and $N$ is (locally) symmetric space. This happens in the following cases:
\begin{center}
\begin{tabular}{|c|c|c|}
\hline
$\fg$ & $\Xi$ & $\kappa_H$ \\
\hline
$\frak{sl}(n+1,\{\R,\C\})$ & $\{1,n\}$ &  $(\alpha_1,\alpha_n)$  \\
\hline
$\frak{su}(q,n+1-q)$ & $\{1,n\}$ &  $(\alpha_1,\alpha_n)$  \\
\hline
$\frak{sp}(2n,\{\R,\C\})$ & $\{1\}$ &  $(\alpha_1,\alpha_2)$  \\
$\frak{sp}(2n,\C)$ & $\{1\}$ &  $(\alpha_1,\alpha_1')$  \\
\hline
\end{tabular}
\end{center}
\item $ker(T_x\pi)\cong \fg_{-2}$ for each $x$ and $N$ is (locally) $\mathbb{Z}_3$--symmetric space. This happens in the following cases:
\begin{center}
\begin{tabular}{|c|c|c|}
\hline
$\fg$ & $\Xi$ & $\kappa_H$ \\
\hline
$\frak{sl}(n+1,\C)$ & $\{1,n\}$ &  $(\alpha_1,\alpha_2),(\alpha_{n},\alpha_{n-1})$  \\
$\frak{sl}(n+1,\C)$ & $\{1,n\}$ &  $(\alpha_1,\alpha_{n'}),(\alpha_{n},\alpha_{1'})$  \\
$\frak{sl}(n+1,\C)$ & $\{1,n\}$ &  $(\alpha_1,\alpha_{n'}),(\alpha_{1},\alpha_{2})$  \\
$\frak{sl}(n+1,\C)$ & $\{1,n\}$ &  $(\alpha_{n},\alpha_{n-1}),(\alpha_{n},\alpha_{1'})$  \\
\hline
$\frak{su}(q,n+1-q)$ & $\{1,n\}$ &  $(\alpha_1,\alpha_2)$  \\
\hline
\end{tabular}
\end{center}
\item $M=N$ is (locally) $\mathbb{Z}_3$--symmetric space. This happens in the following cases:
\begin{center}
\begin{tabular}{|c|c|c|}
\hline
$\fg$ & $\Xi$ & $\kappa_H$ \\
\hline
$\frak{sp}(4,\C)$ & $\{1\}$ &  $(\alpha_1,\alpha_2)$  \\
$\frak{sp}(2n,\C)$ & $\{1\}$ &  $(\alpha_1,\alpha_1')$  \\
\hline
\end{tabular}
\end{center}
\end{itemize}

It depends on the (non--harmonic) component of the curvature of the reduced geometry in $\fg_{\Xi,-1}^*\wedge \fg_{\Xi,-2}^*\otimes \fg_{\Xi,0}$, which of the two possible cases for $\frak{sp}(2n,\C)$ happens.
\end{prop*}
\begin{proof}
All claims follow from the Theorem \ref{7.1}, the classification in the \ref{ApC}, and the fact there are only the following possible components of curvature (in the $\fg_0$--submodules of $\wedge^2 \fg_{\Xi,-}^*\otimes \fq_{\Xi}$) of the reduced geometry:

$\fg_{\Xi,-1}^*\wedge \fg_{\Xi,-1}^*\otimes \fg_{\Xi,0}$ for the eigenvalue $e^{i\phi}$, 

$\fg_{\Xi,-1}^*\wedge \fg_{\Xi,-2}^*\otimes \fg_{\Xi,0}$ for the eigenvalue $\sqrt[3]{1}$ in $\frak{sp}(2n,\C)$--case, 

$\fg_{\Xi,-1}^*\wedge \fg_{\Xi,-1}^*\otimes \fg_{\Xi,-1}$ for the eigenvalue $\sqrt[3]{1}$, and

$\fg_{\Xi,-1}^*\wedge \fg_{\Xi,-2}^*\otimes \fg_{\Xi,-1}$ for eigenvalue $-1$.
\end{proof}

We remark that we investigated $s$--symmetries of the first kind in \cite{LZ2,G2,GZ}. The theory developed in this article provides new results for them in the homogeneous case. Moreover, if there are generalized symmetries of more types, we can strengthen the results of \cite{GZ} in the following way:

\begin{prop*}\label{6.3}
Consider the first case from the Proposition \ref{7.3} and assume there is in addition a generalized symmetry other than the one of order $2$ and the identity. Then the following holds for the parabolic geometries in question:
\begin{enumerate}
\item In the case of Lagrangean contact geometries, there is invariant para--complex structure $\bar{\mathcal{I}}$ on $TN$ induced by the para--complex structure $\mathcal{I}$ on $T^{-1}M$.
\item In the case of CR--geometries, there is invariant complex structure $\bar{\mathcal{I}}$ on $TN$ induced by the complex structure $\mathcal{I}$ on $T^{-1}M$.
\item In the case of complex Lagrangean contact geometries, there are both invariant para--complex and invariant complex structures.
\end{enumerate}
\end{prop*}
\begin{proof}
The existence of $s$--symmetry such that $s\neq-\id_{\fg_{-1}}$ implies that the eigen\-spaces $\fk(j_1)$ and $\fk(j_1^{-1})$ or $\fk(\bar{j_1})$, respectively, are complementary $L$--invariant subspaces in $T_{eL}(K/L)=T_{\pi(x_0)}N$ and provide invariant (para)--complex structure $\bar{\mathcal{I}}$. It is easy to check that it has the claimed properties.
\end{proof}

Let us finally comment briefly the claims in the Theorem \ref{7.1} for several remaining interesting types of parabolic geometries.
\begin{itemize}
\item In the case of (split)--quaternionic contact geometries and their complexifications, $M$ is a (local) reflexion space with a three--dimensional fiber over a (locally) symmetric space $N$. 
\item In the case of the $(2,3,5)$--geometry with $\fg=\frak{g}_{2}(\{2,\mathbb{C}\})$, $M$ is a (local) reflexion space with one--dimensional fiber over a (locally) symmetric space $N$.
\item In the case of the complex free $(3,6)$--distribution, $M$ is $6$--dimensional (locally) $\mathbb{Z}_3$--symmetric space with invariant polarization of $TM$ given by the eigenspaces of $\Ad_{s}$.
\end{itemize}

\appendix

\section{Notation related with the harmonic curvature}\label{ApA}

For a parabolic geometry $(\ba \to M, \om)$, we denote by $\ka$ its curvature, and by $\ka_H$ its harmonic curvature.

As mentioned in the Introduction, we assume that the geometries are regular and normal, which means for homogeneous parabolic geometries that at a point $u_0\in \ba$ (and thus at each point), the curvature $\kappa(u_0)$ viewed as an element of $\bigwedge^2 (\fg/\fp_\Xi)^*\otimes \fg$ has positive homogeneity, and $\partial^*\kappa(u_0)=0$ holds, where $\partial^*$ is the Kostant's codifferential.

Let us remind that the harmonic curvature $\kappa_H(u_0)$ is the projection of the curvature $\kappa(u_0)$ onto the kernel of the Kostant Laplacian $\Box$, see \cite[Section 3.1.12]{parabook}. According to the Kostant's version of the Bott--Borel--Weyl theorem, the kernel $Ker(\Box)$ decomposes as $\fg_{\Xi,0}$--representation into the isotypical components which we represent by ordered pairs $(\alpha_a,\alpha_b)$ meaning that the corresponding $\fp_\Xi$--dominant and $\fp_\Xi$--integral $\fg$--weight is obtained by the affine action of $s_{\alpha_a}s_{\alpha_b}$ on the highest root $\mu^\fg$ of $\fg$, where $s_{\alpha_i}$ denotes the simple reflexion along $\alpha_i$, see \cite[Section 3.2.]{parabook}. We use the notation $(\alpha_a,\alpha_b)$ for the (highest) weight viewed as the element of $H^2(\fp_{\Xi,+},\fg)$, too. The actual (lowest) weights representing indecomposable $\fg_{\Xi,0}$--submodules in $H^2(\fg_{\Xi,-},\fg)$ are obtained via duality, i.e. the lowest weight vector corresponding to $(\alpha_a,\alpha_b)$ is of the form $X^{\alpha_a}\wedge X^{s_{\alpha_b}(\alpha_a)}\otimes X^{-s_{\alpha_a}s_{\alpha_b}(\mu^{\fg})}$, where $X^\alpha$ denotes a root vector for $\alpha$. Then the homogeneity of $(\alpha_a,\alpha_b)$ with respect to $\alpha_i\in \Xi$ for $\mu^\fg=\sum k_i\alpha_i=\sum r_i\lambda_i,$ where $\lambda_i$ denotes the corresponding fundamental weight, can be computed as follows:

\begin{itemize}
\item If $i=a$, then the homogeneity is $-k_a+1+r_a-\langle \alpha_b,\alpha_a\rangle(1+r_b)$, where $\langle\  ,\  \rangle$ is the scalar product induced by the Killing form.
\item If $i=b$, then the homogeneity is $-k_b+1+r_b$.
\item If $i=c$, where $c \neq a$, $c \neq b$, then the homogeneity is $-k_c$.
\end{itemize}

Let us remark that the length $k$ of the grading of $\fg$ given by $\Xi$--heights corresponds to $\sum_{\alpha_i\in\Xi} k_i$.

Let us recall and refine several results of the article \cite{KT}. Firstly, the authors define the sets $I_\mu$ as the sets of roots $\alpha_i \in \Xi$ that satisfy $\langle (\alpha_a,\alpha_b),\alpha_i \rangle=0$ for the highest weight $\mu=(\alpha_a,\alpha_b)$ in $H^2(\fp_{\Xi,+},\fg)$ representing a component of the harmonic curvature. The authors show in \cite[Theorem 3.3.3 and Proposition 3.1.1]{KT} that each set $I_\mu$ restricts the dimensions of projections of $\fk$ into  the associated grading of the filtration of $\fp_\Xi$. In fact, stronger results hold under a condition that is always satisfied in the homogeneous setting:

\begin{prop*}\label{Aprop}
Suppose the Lie algebra of infinitesimal automorphisms with higher order fixed point splits into modules $V_{\Xi,\alpha_i}$ for $\alpha_i\in \Xi$ in the associated grading defined in \cite[Section 2.3.1]{KT}. Then the Proposition 3.1.1 from \cite{KT} holds after the restriction to $V_{\Xi,\alpha_i}$. In particular, $I_\mu$ characterizes the modules $V_{\Xi,\alpha_i}$, where the projection of the Lie algebra of infinitesimal automorphisms with higher order fixed point into the associated grading can be non--trivial.
\end{prop*}
\begin{proof}
Firstly, \cite[Theorem 3.3.3]{KT} gives results compatible with the decomposition to modules $V_{\Xi,\gamma}\subset \fp_{\Xi,+}$. Then the assumption implies that the proof of \cite[Proposition 3.1.1]{KT} can be applied to indecomposable modules of $V_{\Xi,\alpha_i}$ without any change. Then the second claim is a consequence of \cite[Theorem 3.3.3]{KT}.
\end{proof}

Finally, we prove a statement which allows us to compute the set $\Psi$ (defined in the Introduction) explicitly.

\begin{prop*}
The set $\Psi$ equals to the set of all simple roots $\alpha_i$ such that $\langle \alpha_i,(\alpha_a,\alpha_b) \rangle \geq 0$ for all highest weights $(\alpha_a,\alpha_b)$ representing non--trivial components of $\kappa_H$. Moreover,

\begin{enumerate}
\item If $\langle \alpha_a,\alpha_b \rangle \neq 0$, then $\alpha_a$ is the unique simple root $\alpha$ such that $\langle \alpha,(\alpha_a,\alpha_b) \rangle <0$.
\item If $\langle \alpha_a,\alpha_b\rangle = 0$, then $\alpha_a$ and $\alpha_b$ are the unique simple roots $\alpha$ such that $\langle \alpha, (\alpha_a,\alpha_b) \rangle <0$.
\end{enumerate}
\end{prop*}
\begin{proof}
It follows by the general representation theory that the highest weight $\fg$--module of $(\alpha_a,\alpha_b)$ is naturally both $\fp_\Xi$--module and $\fp_{\Xi-\Psi}$--module and thus the component of the harmonic curvature represented by $(\alpha_a,\alpha_b)$ does not take values in $\fg_{\Xi-\Psi,0}$.

For arbitrary simple root $\alpha_i$, the direct computation implies
\begin{align*}
\langle \alpha_i,(\alpha_a,\alpha_b) \rangle&=\langle \alpha_i,s_{\alpha_a}s_{\alpha_b}(\mu^\fg+\rho)-\rho \rangle \\
&=\langle \alpha_i,\mu^\fg\rangle -\left(1+\langle \mu^\fg,\alpha_b\rangle \right)\langle \alpha_i,\alpha_b\rangle\\
&-\left(1+\langle\mu^\fg,\alpha_a\rangle-(1+\langle\mu^\fg,\alpha_b\rangle)\langle \alpha_b,\alpha_a\rangle\right)\langle\alpha_i,\alpha_a\rangle,
\end{align*}
where $\rho$ denotes the lowest weight. Thus $\langle \alpha_i,(\alpha_a,\alpha_b) \rangle\geq 0$ holds for $i \neq a,b$.

If $i=a$, then
\begin{align*}
\langle\alpha_a,(\alpha_a,\alpha_b)\rangle&=-\langle \alpha_a,\mu^\fg\rangle +\left(1+\langle \mu^\fg,\alpha_b\rangle \right)\langle \alpha_a,\alpha_b\rangle-2\\
&=-k_a+(1+k_b)\langle \alpha_a,\alpha_b\rangle-2<0.
\end{align*}

If $i= b$, then
$$\langle\alpha_b,(\alpha_a,\alpha_b)\rangle
=-\langle \alpha_b,\mu^\fg\rangle -2-\left(1+\langle\mu^\fg,\alpha_a\rangle\right)\langle\alpha_b,\alpha_a\rangle+(1+\langle\mu^\fg,\alpha_b\rangle)\langle \alpha_b,\alpha_a\rangle^2.$$
Thus if $\langle \alpha_b,\alpha_a\rangle=0$, then $\langle\alpha_b,(\alpha_a,\alpha_b)\rangle <0$. Otherwise, the last term is greater than the first term, and the third term is greater than the second term in the absolute value, and thus $\langle\alpha_b,(\alpha_a,\alpha_b)\rangle\geq 0$.
\end{proof}

\section{Example}\label{ApB}

Let us illustrate the theory on a particular example. The example also provides a counter--example for the Theorem \ref{main} to hold without considering the set $\Theta(\fk)$.

Let us consider the seventeen--dimensional Lie subgroup $K$ of $Gl(8,\mathbb{R})$ consisting of elements of the form

$$\begin{pmatrix}\pm cos(x_9) &\mp sin(x_9) & 0 & 0 & 0 & 0 & 0 & 0\cr sin(x_9) & cos(x_9) & 0 & 0 & 0 & 0 & 0 & 0\cr 0 & 0 & \pm cosh(x_9) & \pm sinh(x_9) & 0 & 0 & 0 & 0\cr 0 & 0 & sinh(x_9) & cosh(x_9) & 0 & 0 & 0 & 0\cr x_{1} & x_{2} &0 & 0 & x_{10} & x_{11} & 0 & 0\cr x_{3} & x_{4} & 0 & 0 & x_{12} & x_{13} & 0 & 0\cr 0 & 0 & x_{5} & x_{6} & 0 & 0 & x_{14} & x_{15}\cr 0 & 0 & x_{7} & x_{8} & 0 & 0 & x_{16} & x_{17}
\end{pmatrix},$$
where $x_1,\dots,x_{17}\in \mathbb{R}$ and $(x_{10}x_{13}-x_{11}x_{12})(x_{14}x_{17}-x_{15}x_{16})>0$, together with its five--dimensional (solvable) Lie subgroup $H$ consisting of elements of the form

$$h=\begin{pmatrix}\pm 1 &0 & 0 & 0 & 0 & 0 & 0 & 0\cr 0 & 1 & 0 & 0 & 0 & 0 & 0 & 0\cr 0 & 0 & \pm 1 & 0& 0 & 0 & 0 & 0\cr 0 & 0 & 0 & 1& 0 & 0 & 0 & 0\cr 0 & 0 &0 & 0 & x_{10} & x_{11} & 0 & 0\cr 0 & 0 & 0 & 0 & 0 & x_{13} & 0 & 0\cr 0 & 0 & 0 & 0 & 0 & 0 & x_{14} & x_{15}\cr 0 & 0 & 0 & 0 & 0 & 0 & 0 & x_{13}
\end{pmatrix},$$
i.e., $x_{10}x_{13}^2x_{14}>0$.

We will investigate the regular normal homogeneous parabolic geometry on $K/H$ of type $(PGl(6,\mathbb{R}),P_{1,2,5})$ (and thus $\Xi=\{\alpha_1,\alpha_2,\alpha_5\}$ and $\fg=\mathfrak{sl}(6,\mathbb{R})$) given by the extension $(\alpha,\iota)$ of $(K,H)$ to $(PGl(6,\mathbb{R}),P_{1,2,5})$  defined as follows:
We define $\iota(h)$ for $h\in H$ of the above form as:

$$\begin{pmatrix} \pm \frac{1}{\sqrt[6]{x_{10}x_{13}^{2}x_{14}}}&0 & 0 & 0 & 0 & 0 \cr 0 & \frac{1}{\sqrt[6]{x_{10}x_{13}^{2}x_{14}}}&0 & 0 & 0 & 0 \cr 0 & 0 &\frac{x_{10}}{\sqrt[6]{x_{10}x_{13}^{2}x_{14}}} & 0& x_{11} & 0 \cr 0 & 0 & 0 & \frac{x_{14}}{\sqrt[6]{x_{10}x_{13}^{2}x_{14}}}& x_{15} & x_{15} \cr 0 & 0 &0 & 0 & \frac{x_{13}}{\sqrt[6]{x_{10}x_{13}^{2}x_{14}}}& 0 \cr 0 & 0 & 0 & 0 & 0 & \frac{x_{13}}{\sqrt[6]{x_{10}x_{13}^{2}x_{14}}}
\end{pmatrix}.$$
We have chosen $PGl(6,\R)$ instead of $Sl(6,\R)$ in order to allow $\pm$ on the first position, which will turn our to be the usual (geodesic) symmetry for some underlying symmetric space. In particular, $Sl(6,\R)\subset PGl(6,\R)\subset \tG$ in this situation and $Z(\tG_{\Xi,0})\subset PGl(6,\R)$ holds. Further,
elements of $\fk$ are of the form
$$X=\begin{pmatrix}0 &-X_9 & 0 & 0 & 0 & 0 & 0 & 0\cr X_9 & 0 & 0 & 0 & 0 & 0 & 0 & 0\cr 0 & 0 & 0 & X_9 & 0 & 0 & 0 & 0\cr 0 & 0 & X_9 & 0& 0 & 0 & 0 & 0\cr X_{1} & X_{2} &0 & 0 & X_{10} & X_{11} & 0 & 0\cr X_{3} & X_{4} & 0 & 0 & X_{12} & X_{13} & 0 & 0\cr 0 & 0 & X_{5} & X_{6} & 0 & 0 & X_{14} & X_{15}\cr 0 & 0 & X_{7} & X_{8} & 0 & 0 & X_{16} & X_{17}\end{pmatrix}$$
and we define
$$
\alpha(X)=\begin{pmatrix}
C &0 & 0& 0 & 0 & 0 \cr
X_9 & C& 0 & 0 & 0 & 0 \cr
X_1 & X_2 & C+X_{10}& 0& X_{11} & 0 \cr
X_5 & X_6 & 0 & C+X_{14}& X_{15} & X_{15} \cr
X_3 & X_4 &X_{12}& 0 & C+X_{13} & 0 \cr
X_7-X_3 & X_8-X_4 & -X_{12} & X_{16} & X_{17}-X_{13} &C+X_{17}
\end{pmatrix}
,$$ where $X_1, \dots, X_{17} \in \R$ and $C=-\frac{X_{10}+X_{13}+X_{14}+X_{17}}{6}$.

The curvature of this geometry at the origin for $X,Y\in \fk$ as above is of the form

$$\begin{pmatrix}
0 &0 & 0& 0 & 0 & 0 \cr
0 & 0& 0 & 0 & 0 & 0 \cr
0 & X_1Y_9-X_9Y_1 & 0& 0& 0 & 0 \cr
0 &-(X_5Y_9-X_9Y_5) & 0 & 0& 0 & 0 \cr
0&X_3Y_9-X_9Y_3&0& 0& 0 & 0 \cr
0& \bf{-2(X_3Y_9-X_9Y_3)}\rm-(X_7-X_3)Y_9+X_9(Y_7-Y_3) & 0 & 0 & 0 &0
\end{pmatrix}
,$$ where the bold entry corresponds to the harmonic part of the curvature. Thus $\kappa_H$ has non--trivial part in the submodule represented by 
$\mu=(\alpha_1,\alpha_2)$, and $\Phi(\fk)=\emptyset$ and $\Psi=\{\alpha_2,\alpha_5\}$.

It is easy computation to check that the extension satisfies all the properties of the Theorem \ref{extension} except the fact that we are missing the connected components of $\tG$ with the outer automorphism of $\frak{sl}(6,\mathbb{R})$, which does not preserve $\fp_\Xi$. In particular, there are no other infinitesimal automorphisms of this geometry and $\fh\cap\fp_{\Xi,+}=0$. 
Thus the set $\mathcal{J}(\fk)$ consists of elements of the form for $j_2\in \mathbb{R}^\times$:

$$\begin{pmatrix} \pm \frac{1}{\sqrt[3]{j_2^2}}&0 & 0 & 0 & 0 & 0 \cr 0 & \frac{1}{\sqrt[3]{j_2^2}}&0 & 0 & 0 & 0 \cr 0 & 0 &\sqrt[3]{j_2} & 0& 0 & 0 \cr 0 & 0 & 0 & \sqrt[3]{j_2} & 0 & 0 \cr 0 & 0 &0 & 0 &\sqrt[3]{j_2} & 0 \cr 0 & 0 & 0 & 0 & 0 &\sqrt[3]{j_2}
\end{pmatrix},$$
i.e., the eigenvalues $j_i$ on $V_{\Xi,\alpha_i}$ for $\alpha_i\in \Xi$ are $j_1=\pm 1, j_5=1$ and $j_2$ can be arbitrary for the possible generalized symmetries. 
Thus $\Theta(\fk)=\{\alpha_5\}$.

Our theory has the following geometrical consequences for the geometry: It is a simple computation to check that $H$ is not $P$--conjugated to a subgroup of $Q_{\Xi-\Phi(\fk)}=Q_{1,2,5}$. On the other hand, $H\subset Q_{\Lambda(\fk)}=Q_{1,2}$ holds consistently with our theory. Since eigenvalues of all generalized symmetries on $\fg_{-k}$ equal to $-j_2$, we can apply the Theorem \ref{4.5} for $L_{\Theta(\fk)\cap\Psi}=L_{5}$, which consists of  elements of the form

$$\begin{pmatrix}\pm 1 &0 & 0 & 0 & 0 & 0 & 0 & 0\cr 0& 1 & 0 & 0 & 0 & 0 & 0 & 0\cr 0 & 0 & \pm 1 & 0 & 0 & 0 & 0 & 0\cr 0 & 0 & 0 & 1 & 0 & 0 & 0 & 0\cr 0 & 0 &0 & 0 & x_{10} & x_{11} & 0 & 0\cr 0 & 0 & 0 & 0 & x_{12} & x_{13} & 0 & 0\cr 0 & 0 & 0 & 0 & 0 & 0 & x_{14} & x_{15}\cr 0 & 0 & 0 & 0 & 0 & 0 & x_{16} & x_{17}\end{pmatrix}.$$

Since $L_5$ is reductive, $K/H$ is open in the correspondence space to the parabolic geometry of type $(PGl(6,\mathbb{R}),P_{1,2})$ given by the extension $(\bar \alpha, \bar \iota)$ of $(K,L_5)$ to $(PGl(6,\mathbb{R}),P_{1,2})$ defined as follows: The map $\bar \iota$ is $\frac{1}{\sqrt[6]{x_{10}x_{13}x_{14}x_{17}}}$ multiple of the restriction of $L_5$ to the bottom right block, and $\bar \alpha$ has the following form for $X\in \fk$ as above:

$$\bar \alpha(X)=\begin{pmatrix}
C &0 & 0& 0 & 0 & 0 \cr
X_9 & C& 0 & 0 & 0 & 0 \cr
X_1 & X_2 & C+X_{10}& X_{11}& 0 & 0 \cr
X_3 & X_4 & X_{12} & C+X_{13}& 0 & 0 \cr
X_5 & X_6 &0& 0 & C+X_{14} & X_{15} \cr
X_7 & X_8 &0 &0 & X_{16} &C+X_{17}
\end{pmatrix}
,$$ where $C=-\frac{X_{10}+X_{13}+X_{14}+X_{17}}{6}$. In particular, $K/H$ is the $K$--orbit of $pP_{1,2,5}$ for $p\in P_{1,2}$ of the form

$$p=\begin{pmatrix}
1 &0 & 0& 0 & 0 & 0 \cr
0 & 1& 0 & 0 & 0 & 0 \cr
0 & 0& 1& 0& 0 & 0 \cr
0& 0 & 0 & 0& 1 & 0 \cr
0 & 0 &0& 1& 0 & 0 \cr
0 &0 &0 &0 & 1 &1
\end{pmatrix}
.$$

Thus $\Ad_p(\alpha)=\bar \alpha$ and $p\iota(h)p^{-1}=\bar \iota(h)$. Now, $L_5\subset G_{\{1,2\},0}$ holds and the whole curvature of this geometry is harmonic, which means that the $L_5$--invariant complement of $\fl_5$ in $\fk$ provides the integrable distribution in $T(K/H)$. Let us point out that this is not the case when the whole curvature is not harmonic.

Now, we can apply the Theorem \ref{5.1} on the geometry on $K/L_5$, because $L_5\subset G_{\{1,2\},0}$. So there is the invariant Weyl connection on $K/L_5$, which is covered by a class of almost $\{1,2\}$--invariant Weyl connections on $K/H$.

If $ s\in \mathcal{J}(\fk)$ satisfies $j_2=1$, then $\Psi(1) \cap \Xi=\{2\}$, and we can apply the Theorem \ref{4.4} for the subgroup $L=L_{2,5}$ consisting of the elements of the form

$$\begin{pmatrix}\pm 1 &0 & 0 & 0 & 0 & 0 & 0 & 0\cr 0& 1 & 0 & 0 & 0 & 0 & 0 & 0\cr 0 & 0 & \pm 1 & 0 & 0 & 0 & 0 & 0\cr 0 & 0 & 0 & 1 & 0 & 0 & 0 & 0\cr 0 & x_2 &0 & 0 & x_{10} & x_{11} & 0 & 0\cr 0 & x_4 & 0 & 0 & x_{12} & x_{13} & 0 & 0\cr 0 & 0 & 0 & x_6 & 0 & 0 & x_{14} & x_{15}\cr 0 & 0 & 0 & x_8 & 0 & 0 & x_{16} & x_{17}\end{pmatrix}.$$

The pair $(K,L_{2,5})$ is reductive, and $K/L_5$ and $K/H$ are open in the correspondence spaces to the projective geometry given by the extension $(\bar \alpha, \hat \iota)$ of $(K,L_5)$ to $(PGl(6,\mathbb{R}),P_{1})$ defined as follows: The map $\hat \iota$ is $\frac{1}{\sqrt[6]{x_{10}x_{13}x_{14}x_{17}}}$ multiple of the restriction of $L_{2,5}$ to the bottom right block after the moving the $x_2,x_4$ to position that is two columns to the right, and $\bar \alpha$ is the same as above. 
In particular, $K/H$ is again the $K$--orbit of $pP_{1,2,5}$ and $K/L_5$ is the $K$--orbit of $eP_{1,2}$.

Moreover, the projective geometry on $K/L_{2,5}$ is symmetric. Thus $K/L_{2,5}$ is non--effective symmetric space, because the Lie algebra of the group generated by symmetries corresponds to the $L_5$--invariant complement of $\fl_5$ in $\fk$.

This geometry has further geometric properties, which does not have to occur in the general situation.
Firstly, the distribution $T^{\Lambda(\fk),-}(K/H)$ is integrable and corresponds to the Lie algebra of the group generated by generalized symmetries.

Further, for $\{\alpha_2\}\subset \Psi$, the subgroup $L_2$ consisting of the elements of the form
$$\begin{pmatrix}\pm 1 & 0 & 0 & 0 & 0 & 0 & 0 & 0\cr 0 & 1 & 0 & 0 & 0 & 0 & 0 & 0\cr 0 & 0 & \pm 1 & 0 & 0 & 0 & 0 & 0\cr 0 & 0 & 0 & 1 & 0 & 0 & 0 & 0\cr 0 & x_{2} &0 & 0 & x_{10} & x_{11} & 0 & 0\cr 0 & x_{4} & 0 & 0 & 0 & x_{13} & 0 & 0\cr 0 & 0 & 0 & x_{6} & 0 & 0 & x_{14} & x_{15}\cr 0 & 0 & 0 & x_{4} & 0 & 0 & 0 & x_{13}\end{pmatrix}$$
is closed in $K$, and we can apply the analogy of the Proposition \ref{4.4} for this situation. Namely, there is the Lagrangean contact geometry on $K/L_2$ given by the extension $(\alpha,\check \iota)$ of $(K,L_2)$ to $(PGl(6,\mathbb{R}),P_{1,5})$ defined as follows: We define $\check \iota(h)$ for $h\in L_2$ of the above form as:

$$\begin{pmatrix} \pm \frac{1}{\sqrt[6]{x_{10}x_{13}^{2}x_{14}}}&0 & 0 & 0 & 0 & 0 \cr 0 & \frac{1}{\sqrt[6]{x_{10}x_{13}^{2}x_{14}}}&0 & 0 & 0 & 0 \cr 0 & x_2 &\frac{x_{10}}{\sqrt[6]{x_{10}x_{13}^{2}x_{14}}} & 0& x_{11} & 0 \cr 0 & x_6 & 0 & \frac{x_{14}}{\sqrt[6]{x_{10}x_{13}^{2}x_{14}}}& x_{15} & x_{15} \cr 0 & x_4 &0 & 0 & \frac{x_{13}}{\sqrt[6]{x_{10}x_{13}^{2}x_{14}}}& 0 \cr 0 & 0 & 0 & 0 & 0 & \frac{x_{13}}{\sqrt[6]{x_{10}x_{13}^{2}x_{14}}}
\end{pmatrix}.$$

This geometry has only generalized symmetries of one type, and the generalized symmetries are the lifts of the symmetries of the above underlying symmetric projective geometry. So this geometry satisfies $\Theta(\fk)=\Psi(1)=\{\alpha_5\}$, and $K/L_2$ is the $K$--orbit of $pP_{1,5}$ in the correspondence space to the projective geometry on $K/L_{2,5}$ for $p$ as above.

Finally, we can consider the subgroup $L_{1,5}$, which is the centralizer of $ s\in Z(G_0)\cap H$ with eigenvalues $j_1=1, j_2=-1, j_5=1$ in $K$, consisting of the elements of the form

$$\begin{pmatrix}\pm cos(x_9) &\mp sin(x_9) & 0 & 0 & 0 & 0 & 0 & 0\cr sin(x_9) & cos(x_9) & 0 & 0 & 0 & 0 & 0 & 0\cr 0 & 0 & \pm cosh(x_9) & \pm sinh(x_9) & 0 & 0 & 0 & 0\cr 0 & 0 & sinh(x_9) & cosh(x_9) & 0 & 0 & 0 & 0\cr 0 & 0 &0 & 0 & x_{10} & x_{11} & 0 & 0\cr 0 & 0 & 0 & 0 & x_{12} & x_{13} & 0 & 0\cr 0 & 0 & 0 & 0 & 0 & 0 & x_{14} & x_{15}\cr 0 & 0 & 0 & 0 & 0 & 0 & x_{16} & x_{17}\end{pmatrix}.$$
Then $K/L_{1,5}$ is a symmetric space, but since there are entries containing $X_9, Y_9$ in the curvature, the parabolic geometry of type $(PGl(6,\mathbb{R}),P_{1,2,5})$ on $M$ does not descend to a  parabolic geometry of type $(PGl(6,\mathbb{R}),P_{2})$ on this underlying symmetric space.

\section{Tables}\label{ApC}

We present here tables classifying all possible generalized symmetries of non--flat parabolic geometries with $\fg$ simple. We recall that there is the construction in \cite{GZ2}, that allows to construct explicit examples of such geometries. The data in the tables can be divided into two parts. The first one describes the information on the parabolic geometry and consists of the following data:

\begin{itemize}
\item[$\fg$] distinguishes the simple Lie algebra $\fg$, where $n\geq |\Xi|$ determines the rank and $q>0$ the signature. We assume that $n>1$ for type $A$, $n>1$ for type $C$, $n>2$ for type $B$, $n>3$ for type $D$, and $n>6$ for type $BD$.
\item[$\Xi$] distinguishes the parabolic subalgebra $\fp_\Xi$, where the parameter $p$ is such that all simple roots in $\Xi$ defines a simple restricted root. In the cases dealing with mixed types of curvature, where conjugate roots have different homogeneity, we list the conjugate roots (indicated by ') in $\Xi$, too.
\item[param.] contains additional restrictions on parameters $p,q,n$.
\item[$\mu$] represents by the pair  $(\alpha_a,\alpha_b)$ the $\fg_{\Xi,0}$--submodule  to which $\kappa_H$ can have non--trivial projection.
\item[homog.] is the tuple characterizing homogeneity of $(\alpha_a,\alpha_b)$ with respect to $\alpha_i\in \Xi$ ordered in the same way as $\Xi$.
\item[$I_\mu$] characterizes all simple roots of $\Xi$ that can be contained in $\Phi(\fk)$ for this $\mu$, see  \ref{ApA}. Additional restrictions on parameters are listed to distinguish special situations.
\item[all info] indicates, when it is convenient to put all the information in one column. The information in rows (if they are not missing) are ordered as above.
\end{itemize}

Then for the given types of parabolic geometries, the second part of the data describes the possible types of generalized symmetries by characterizing the possible eigenvalues $j_i$ of $\Ad_s$ on the simple restricted root spaces $\fg_{\alpha_i}$ for $\alpha_i\in \Xi$. Let us recall that in this setting,  the only restrictions on the possible eigenvalues $j_{i}$ are the following:
\begin{itemize}
\item  if $\alpha_{i_1}$ and $\alpha_{i_2}$ are complex conjugated, then $j_{i_1}=\bar j_{i_2}$,
\item  if  $a_i$ denotes the homogeneity of $(\alpha_a,\alpha_b)$ with respect to $\alpha_i\in \Xi$, then
$$\prod_{i \in \Xi} (j_i)^{a_i}=1.$$
\end{itemize}
The first restriction follows from \cite[Proposition 3.2]{GZ2}, and the second one from the definition of homogeneity. We split the different cases depending  how the $1$--eigenspace $\fm$ of $\Ad_s$ in $\fp_{\Xi,+}$ looks like.

\begin{itemize}
\item[$j_{i_l}$] the eigenvalue of $\Ad_{s}$ on the root space of $l$--th simple root in $\Xi$. In the case of non--mixed curvature, the complex eigenvalues are possible only for $\alpha_i$ conjugated to another simple root and we write the conjugate eigenvalue only for $|\Xi|>1$. In the case of mixed curvature, we write the eigenvalues in the form $e^{r+i\phi}$ for $r\in \mathbb{R}$ and $\phi\in \mathbb{S}^1$. In both cases, $\sqrt[a]{1}$ represents $e^{\frac{2k\pi}{a}}$ for $k\in \mathbb{Z}$ not divisible by $a$.
\item[$\frak{m}$] lists the roots representing the irreducible $\fg_{\Xi,0}$--submodules of $\fp_{\Xi,+}$ in the $1$--eigenspace $\frak{m}$ of $\Ad_s$. We use $\dots$ to shorten the entry, when the weight is determined up to $\fg_{\Xi,0}$--weights.
\end{itemize}

\begin{tabular}{|c|c|c||c|c|}

\hline

$\fg$ & $\Xi$ & $\mu$& $j_{i_1}$&$\frak{m}$ \\

\hline

\hline


$\frak{sl}(n+1,\mathbb{C})$ & $\{1\}$ & $(\alpha_1,\alpha_{1'})$&$-e^{i\phi}$& \\

\hline

$\frak{sl}(n+1,\{\mathbb{R, C}\}), n>2$ & $\{1\}$& $(\alpha_1,\alpha_2)$&$-1$& \\
\hline

$\frak{sl}(n+1,\{\mathbb{R, C, H}\}), n>2$ &$\{2\}$ & $(\alpha_2,\alpha_1)$&$-1$& \\

\hline
$\frak{sp}(4,\mathbb{C})$ & $\{1\}$ & $(\alpha_1,\alpha_2)$&$\sqrt[3]{1}$&\\

\hline

$\frak{sp}(2n,\mathbb{C})$ & $\{1\}$& $(\alpha_1,\alpha_1')$ &$-1$& $2\alpha_1+\dots$\\

&& &$-e^{i\phi}$& $ $\\

\hline

$\frak{sp}(2n,\{\mathbb{R,C}\}), n>2$ & $\{1\}$ & $(\alpha_1,\alpha_2)$&$-1$& $2\alpha_1+\dots$ \\

\hline

$\frak{sp}(2n,\{\mathbb{R, C}\}), n>2$ & $\{2\}$ & $(\alpha_2,\alpha_1)$& $-1$& $2\alpha_2+\dots$ \\

\hline

$\frak{sp}(q,n-q), n>2$ & $\{2\}$ & $(\alpha_2,\alpha_1)$&$-1$& $2\alpha_2+\dots$\\

\hline

$\frak{so}(7,\mathbb{C})$ & $\{3\}$ &$(\alpha_3,\alpha_2)$&$\sqrt[3]{1}$&\\

\hline

$\frak{so}(q,n-q)$,$\frak{so}(n,\mathbb{C})$ & $\{1\}$ & $(\alpha_1,\alpha_2)$&$-1$& \\

\hline

$\frak{g}_{2}(\{2,\mathbb{C}\})$ & $\{1\}$&$(\alpha_1,\alpha_2)$&$-1$&$2\alpha_1+\alpha_2$\\

&&&$\sqrt[4]{1}$&$ $\\

\hline

\end{tabular}

\begin{tabular}{|c|c|c|c||c|}

\hline

$\fg$ & $\Xi$ & $\mu$& homog.& $j_{i_1}$ \\

\hline

\hline

$\frak{sl}(n+1,\mathbb{C})$ & $\{p,p'\}$& $(\alpha_{p'},\alpha_{p+1'})$ &$(-1,2)$ &$\sqrt[3]{1}$ \\

\hline

$\frak{sp}(2n,\mathbb{C})$ & $\{n-1,n-1'\}$& $(\alpha_{n-1'},\alpha_{n'})$&$(-2,3)$&$\sqrt[5]{1}$\\

\hline

$\frak{sp}(2n,\mathbb{C})$ & $\{n,n'\}$& $(\alpha_{n'},\alpha_{n-1'})$&$(-1,2)$&$\sqrt[3]{1}$\\

\hline

$\frak{so}(2n+1,\mathbb{C})$ & $\{n,n'\}$& $(\alpha_{n'},\alpha_{n-1'})$ & $(-2,3)$&$\sqrt[5]{1}$ \\

\hline

$\frak{so}(2n,\mathbb{C})$ & $\{n,n'\}$& $(\alpha_{n'},\alpha_{n-2'})$ & $(-1,2)$&$\sqrt[3]{1}$ \\

\hline

$\frak{so}(n,\mathbb{C})$ & $\{1,1'\}$& $(\alpha_{1'},\alpha_{2'})$ & $(-1,2)$&$\sqrt[3]{1}$ \\

\hline

$\frak{e}_{6}(\mathbb{C})$ & $\{1,1'\}$&$(\alpha_{1'},\alpha_{2'})$&$(-1,2)$ &$\sqrt[3]{1}$\\

\hline

$\frak{e}_{7}(\mathbb{C})$ & $\{1,1'\}$&$(\alpha_{1'},\alpha_{2'})$&$(-1,2)$ &$\sqrt[3]{1}$\\

\hline

\end{tabular}

\begin{tabular}{|c|c|c||c|c|c|}

\hline

$\fg$ & $\Xi$ & $I_\mu$ & $j_{i_1}$ & $j_{i_2}$ & $\frak{m}$\\

param.& $\mu$ & homog. & & & \\

\hline \hline


$\frak{sl}(n+1,\{\R,\C\})$ & $\{1,2\}$ & & $-1$ & $-1$ & $\alpha_1+\alpha_2$ \\

$n>2$& $(\alpha_1,\alpha_2)$ &$(2,0)$ & $-1$& $1$ &$\alpha_2$\\

& & & $1$& $j_2$ &$\alpha_1$\\

& & & $-1$& $j_2$ &\\

\hline

$\frak{sl}(n+1,\{\R,\C\})$ & $\{1,2\}$ & $\alpha_1$ & $1$ & $-1$ & $\alpha_1$ \\

$n>2$& $(\alpha_2,\alpha_1)$ &$(1,2)$ & $j_2^{-2}$& $j_2$ & \\

\hline

$\frak{sl}(n+1,\{\R,\C\})$ & $\{1,n\}$ & & $j_1$ & $j_1^{-1}$ & $\alpha_1+ \dots + \alpha_n$ \\

$n>2$& $(\alpha_1,\alpha_n)$ &$(1,1)$ &$ $& $ $ & \\

\hline

$\frak{su}(q,n+1-q)$ & $\{1,n\}$ & & $e^{i\phi}$ & $e^{-i\phi}$ & $\alpha_1+ \dots + \alpha_n$ \\

$n>2$& $(\alpha_1,\alpha_n)$ &$(1,1)$ &$ $& $ $ & \\

\hline

$\frak{sl}(n+1,\{\R,\C\})$ & $\{1,p\}$ & $\alpha_p, p\neq 3,n-1,n$& $-1$ & $1$ & $\alpha_p$ \\

$2<p$& $(\alpha_1,\alpha_2)$ &$(2,-1)$ &$\sqrt[3]{1}$& $\sqrt[3]{1}^2$ & $\alpha_1+\dots+\alpha_p$\\

& & &$j_1$&$j_1^2$& \\

\hline

$\frak{su}(q,n+1-q)$ & $\{1,n\}$ & &$\sqrt[3]{1}$& $\sqrt[3]{1}^2$ & $\alpha_1+\dots+\alpha_n$\\

$n>2$& $(\alpha_1,\alpha_2)$ &$(2,-1)$ &$ $& $ $ & \\

\hline

$\frak{sl}(n+1,\{\R,\C,\mathbb{H}\})$ & $\{2,p\}$ & $\alpha_p, n>p>3$& $-1$ & $1$ & $\alpha_p$ \\

$2<n,2<p$& $(\alpha_2,\alpha_1)$ &$(2,-1)$ &$\sqrt[3]{1}$&$\sqrt[3]{1}^2$&$\alpha_2+ \dots+\alpha_p$ \\

& & &$j_2$&$j_2^2$& \\

\hline

$\frak{su}(q,n-q+1)$ & $\{2,n-1\}$ & &$\sqrt[3]{1}$&$\sqrt[3]{1}^2$&$\alpha_2+ \dots+\alpha_{n-1}$ \\

$n>q>1$& $(\alpha_2,\alpha_1)$ &$(2,-1)$ & $ $ & $ $ & $ $ \\

\hline

$\frak{sl}(n+1,\{\mathbb{R, C}\})$& $\{1,p\}$& & $1$& $j_p$& $\alpha_1$ \\

$n>3$& $(\alpha_1,\alpha_p)$&$(1,0)$& $ $ & $ $ & $ $\\

\hline

$\frak{sl}(n+1,\{\mathbb{R, C}\})$& $\{p,p+1\}$&$\alpha_p$& $j_p$ & $1$ & $ \alpha_{p+1}$ \\

$n-1>p>1$& $(\alpha_{p+1},\alpha_p)$&$(0,1)$& $ $ & $ $ & $ $\\

\hline

\end{tabular}

\begin{tabular}{|c|c|c||c|c|c|}

\hline

$\fg$ & $\Xi$ & $I_\mu$ & $j_{i_1}$ & $j_{i_2}$ & $\frak{m}$\\

param.& $\mu$ & homog. & & & \\

\hline \hline


$\frak{sp}(4,\{\mathbb{R},\C\})$&$\{ 1,2\}$& & $\sqrt[3]{1}$&$1$& $\alpha_2$\\

& $(\alpha_1,\alpha_2)$& $(3,0)$ &$1$&$j_1$& $ \alpha_1$ \\

& & &$\sqrt[3]{1}$&$\sqrt[3]{1}^2$& $\alpha_1+\alpha_2$ \\

& & &$\sqrt[3]{1}$&$\sqrt[3]{1}$&$2\alpha_1+\alpha_2$ \\

& & &$\sqrt[3]{1}$&$j_2$& \\

\hline

$\frak{sp}(2n,\{\R,\C\})$ & $\{1,2\}$ & $\alpha_1$ & $1$ & $-1$ & $\alpha_1, \alpha_1+2\alpha_2+ \dots,2\alpha_1+\dots$ \\

$n>2$ & $(\alpha_2,\alpha_1)$ &$(1,2)$& $j_2^{-2}$ & $j_2$ & $\alpha_1+2\alpha_2+ \dots $ \\

\hline

$\frak{sp}(2n,\{\R,\C\})$ & $\{1,n\}$ & $\alpha_n, n>3$ & $-1$ & $1$ & $\alpha_n,2\alpha_1+\dots,2\alpha_2+\dots $ \\

$n>2$ & $(\alpha_1,\alpha_2)$ &$(2,-1)$& $\sqrt[3]{1}$ & $\sqrt[3]{1}^2$ & $\alpha_1+ \dots+\alpha_n$ \\

& & & $\sqrt[4]{1}$ & $\sqrt[4]{1}^2$ & $2\alpha_1+\dots $ \\

& & & $j_1$&$j_1^2$& \\

\hline

$\frak{sp}({n \over 2},{n \over 2})$& $\{2,n\}$ & $\alpha_n, n>3$ & $-1$ & $1$ & $\alpha_n,2\alpha_2+\dots,\alpha_1+2\alpha_2+\dots$ \\

$\frak{sp}(2n,\{\R,\C\})$ & $(\alpha_2,\alpha_1)$ &$(2,-1)$& $\sqrt[3]{1}$ & $\sqrt[3]{1}^2$ & $\alpha_2+\dots+ \alpha_n$ \\

$n>2$ & & & $\sqrt[4]{1}$ & $\sqrt[4]{1}^2$ & $2\alpha_1+\dots$ \\

& & & $j_2$&$j_2^2$& \\

\hline

$\frak{sp}(2n,\{\R,\C\})$ & $\{1,2\}$ & & $-1$ & $1$ & $\alpha_2,2\alpha_1+ \dots$ \\

$n>2$ & $(\alpha_1,\alpha_2)$ &$(2,-1)$& $\sqrt[3]{1}$ & $\sqrt[3]{1}^2$ & $\alpha_1+\alpha_2,2\alpha_1+\dots$ \\

& & & $\sqrt[5]{1}$ & $\sqrt[5]{1}^2$ & $\alpha_1+2\alpha_2+\dots$ \\

& & & $\sqrt[6]{1}$ & $\sqrt[6]{1}^2$ & $2\alpha_1+\dots$ \\

& & & $j_1$&$j_1^2$& \\

\hline

$\frak{sp}(2n,\{\mathbb{R, C}\})$& $\{n-1,n\}$& &$1$&$j_n$&$\alpha_{n-1} $ \\

$n>2$& $(\alpha_{n-1},\alpha_n)$&$(1,0)$&$ $&$ $&$ $ \\

\hline

$\frak{sp}(2n,\{\mathbb{R, C}\})$& $\{1,n\}$& &$1$&$j_n$& $\alpha_1$ \\

$n>2$& $(\alpha_1,\alpha_n)$&$(1,0)$&&& \\

\hline

\end{tabular}

\begin{tabular}{|c|c|c||c|c|c|}

\hline

$\fg$ & $\Xi$ & $I_\mu$ & $j_{i_1}$ & $j_{i_2}$ & $\frak{m}$\\

param.& $\mu$ & homog. & & & \\

\hline \hline

$\frak{so}(3,4), \frak{so}(7,\C)$&$\{ 1,3\}$& & $1$&$\sqrt[3]{1}$& $\alpha_1$\\

& $(\alpha_3,\alpha_2)$& $(-1,3)$ &$\sqrt[4]{1}^3$&$\sqrt[4]{1}$& $\alpha_1+\alpha_2+\alpha_3$ \\

& & &$\sqrt[5]{1}^3$&$\sqrt[5]{1}$& $\alpha_1+2\alpha_2 $ \\

& & &$j_3^3$&$j_3$&$ $ \\

\hline

$\frak{so}(3,4), \frak{so}(7,\C)$&$\{ 2,3\}$& & $1$& $\sqrt[3]{1}$&$\alpha_2$\\

& $(\alpha_3,\alpha_2)$& $(0,3)$ &$-1$&$1$& $\alpha_3,\alpha_1+2\alpha_2+2\alpha_3$ \\

& & &$j_2$&$1$& $ \alpha_3$ \\

& & &$\sqrt[3]{1}^2$&$\sqrt[3]{1}$&$\alpha_1+2\alpha_2+2\alpha_3$ \\

& & & $\sqrt[3]{1}$&$\sqrt[3]{1}$& $\alpha_2+2\alpha_3$\\

& & &$j_2$&$\sqrt[3]{1}$& \\

\hline

$\frak{so}(3,5)$&$\{ 3,4\}$& &$\sqrt[3]{1}$&$\sqrt[3]{1}^2$& $\alpha_1+\dots+\alpha_4$ \\

& $(\alpha_3,\alpha_2)$& $(2,-1)$ & && \\

\hline

$\frak{so}(n,n), \frak{so}(2n,\C)$&$\{ 1,n\}$&$\alpha_n,n>4$ & $-1$& $1$&$\alpha_n$\\

& $(\alpha_1,\alpha_2)$& $(2,-1)$ &$\sqrt[3]{1}$&$\sqrt[3]{1}^2$& $\alpha_1+\dots+\alpha_n$ \\

& & & $j_1$&$j_1^2$& \\

\hline

$\frak{so}(q,n-q)$&$\{ 1,2\}$&$\alpha_2$ & $1$& $-1$&$ \alpha_1,\alpha_1+2\alpha_2+\dots$\\

$\frak{so}(n,\C)$ & $(\alpha_1,\alpha_2)$& $(2,0)$ &$-1$&$1$& $\alpha_2$ \\

$q>1$ & & & $-1$ & $-1$ & $\alpha_1+\alpha_2$ \\

& & &$-1$&$ \sqrt[4]{1}$&$ \alpha_1+2\alpha_2+\dots$ \\

& & & $1$&$j_2$& $\alpha_1$\\

& & & $-1$&$j_2$& \\

\hline

$\frak{so}(q,n-q)$,$\frak{so}(n,\mathbb{C})$& $\{1,2\}$ &&$j_1$&$1$&$\alpha_2$ \\

$q>1$& $(\alpha_2,\alpha_1)$&$(0,1)$ &&& \\

\hline

$\frak{so}(q,n-q)$,$\frak{so}(n,\mathbb{C})$& $\{2,3\}$ &$\alpha_2$&$-1$&$1$& $\alpha_3, \alpha_1+2\alpha_2+\dots$\\

$n>8,q>2$& $(\alpha_3,\alpha_2)$&$(0,1)$ &$j_2$&$1$& $\alpha_3$ \\

\hline

$\frak{g}_2(\{2,\C\})$ & $\{1,2\}$ & & $\sqrt[3]{1}$ & $1$ & $\alpha_2, 3\alpha_1+2\alpha_2 $ \\

& $(\alpha_1,\alpha_2)$ &$(4,0)$ & $-1$ & $1$ & $\alpha_2, 2\alpha_1+\alpha_2$ \\

& & & $1$ & $j_2$& $\alpha_1$ \\

& & & $\sqrt[4]{1}$ & $1$ & $\alpha_2$ \\

& & & $1$ & $-1$ & $\alpha_1, 3\alpha_1+2\alpha_2$ \\

& & & $-1$ & $-1$ & $\alpha_1+\alpha_2, 3\alpha_1+\alpha_2$ \\

& & & $\sqrt[4]{1}$ & $j_2$& \\

\hline

\end{tabular}

\begin{tabular}{|c|c||c|c|c|c|c|}

\hline

$\fg$ & $\mu$ & $j_{i_1}$& $j_{i_2}$& $j_{i_3}$& $j_{i_4}$&$\frak{m}$ \\

param. & $\Xi$ & & & && \\

$I_\mu$ & homog.& && && \\

\hline

\hline

$\frak{sl}(n+1,\mathbb{C})$ & $ (\alpha_{(p+1)'},\alpha_{p'})$&$1$& $\sqrt[3]{1}$& $1$& $\sqrt[3]{1}^2$&$\alpha_p$ \\

& $\{p,p+1,p',(p+1)'\}$& $\pm e^{r}$& $1$& $\pm e^{r}$& $1$& $\alpha_{p+1}$ \\

$\alpha_p$& $(-1,-1,1,2)$& $e^{r+i\phi}$& $e^{-2/3i\phi}$& $e^{r-i\phi}$& $e^{2/3i\phi}$& \\

\hline

$\frak{sl}(n+1,\mathbb{C})$ & $(\alpha_{1},\alpha_{p'})$&$1$&$\pm e^{r}$& $1$& $\pm e^{r}$& $\alpha_1$ \\

& $\{1,p,1',p'\}$ &$\sqrt[3]{1}$&$\sqrt[3]{1}^2$&$\sqrt[3]{1}^2$& $\sqrt[3]{1}$& $\alpha_1+\dots+\alpha_p$\\

& $(1,-1,0,1)$&$e^{-2i\phi}$& $e^{r-i\phi}$& $e^{2i\phi}$& $e^{r+i\phi}$& \\

\hline

$\frak{sl}(n+1,\mathbb{C})$ & $(\alpha_{1},\alpha_{1'})$&$e^{i\phi}$& $1$& $e^{-i\phi}$& $1$&$\alpha_p$ \\

& $\{1,p,1',p'\}$&$e^{r+i\phi}$& $e^{2r}$& $e^{r-i\phi}$& $e^{2r}$&\\

$\alpha_p$&$(1,-1,1,0)$&&&&&\\

$n>p>2$&&&&&& \\

\hline

$\frak{sp}(2n,\mathbb{C})$ & $(\alpha_{(n-1)'},\alpha_{n'})$&$1$&$\pm e^{r}$&$1$&$\pm e^{r}$ &$\alpha_{n-1}$\\

& $\{n-1,n,(n-1)',n'\}$&$\sqrt[5]{1}^3$&$1$&$\sqrt[5]{1}^2$&$1$ &$\alpha_{n}$\\

&$(-2,-1,3,1)$&$\sqrt[3]{1}$&$\sqrt[3]{1}^2$&$\sqrt[3]{1}^2$&$\sqrt[3]{1}$ &$\alpha_{n-1}+\alpha_{n}$\\

&&$e^{-2/5i\phi}$& $e^{r+i\phi}$& $e^{2/5i\phi}$& $e^{r-i\phi}$& \\

\hline

$\frak{sp}(2n,\mathbb{C})$ & $(\alpha_{1},\alpha_{1'})$ &$-1$& $1$& $-1$& $1$&$\alpha_n,2\alpha_1+\dots$ \\

& $\{1,n,1',n'\}$ &$e^{i\phi}$& $1$& $e^{-i\phi}$& $1$&$\alpha_{n}$ \\

$\alpha_n$& $(1,-1,1,0)$&$e^{r+i\phi}$& $e^{2r}$& $e^{r-i\phi}$& $e^{2r}$& \\

\hline

$\frak{sp}(2n,\mathbb{C})$ & $(\alpha_{1},\alpha_{n'})$ &$\sqrt[5]{1}^3$&$\sqrt[5]{1}^4$&$\sqrt[5]{1}^2$& $\sqrt[5]{1}$& $2\alpha_1+\dots$\\

& $\{1,n,1',n'\}$ &$\sqrt[3]{1}$&$\sqrt[3]{1}^2$&$\sqrt[3]{1}^2$& $\sqrt[3]{1}$& $\alpha_1+\dots+\alpha_n$\\

& $(1,-1,0,1)$ & $\pm e^{r}$& $1$& $\pm e^{r}$& $1$& $\alpha_{n}$ \\

&&$e^{-2i\phi}$& $e^{r-i\phi}$& $e^{2i\phi}$& $e^{r+i\phi}$& \\

\hline

\end{tabular}

\begin{tabular}{|c||c|c|c|c|}

\hline

all info & $j_{i_1}$& $j_{i_2}$& $j_{i_3}$&$\frak{m}$ \\

\hline

\hline

$\frak{sl}(n+1,\{\mathbb{R, C}\})$ & $1$& $-1$&$1$&$\alpha_1,\alpha_p$ \\

$\{1,2,p\}$ &$-1 $& $ 1$& $ -1$& $\alpha_2,\alpha_1+\dots+\alpha_p $\\

$(\alpha_2,\alpha_1)$&$-1 $& $ -1$& $ -1$& $\alpha_1+\alpha_2,\alpha_2+\dots+\alpha_p $\\

$(1,2,-1)$&$1 $& $\sqrt[3]{1} $& $\sqrt[3]{1}^2 $& $ \alpha_1,\alpha_2+\dots+\alpha_p,\alpha_1+\dots+\alpha_p $\\

$\alpha_1,$&$ j_2^{-2}$& $j_2 $& $1 $& $\alpha_p $\\

$\alpha_p, n>p> 3$&$1 $& $j_2 $& $j_2^2 $& $ \alpha_1 $\\

&$j_1 $& $ 1$& $ j_1$& $ \alpha_2 $\\

&$ j_2^{-3}$& $ j_2$& $j_2^{-1} $& $\alpha_2+\dots+\alpha_p $\\

&$ j_p^{-3}$& $j_p^2 $& $j_p $& $\alpha_1+\dots+\alpha_p $\\

&$ j_1$& $ j_1^{-1}$& $j_1^{-1}$& $\alpha_1+\alpha_2 $\\

&$j_1 $& $j_2 $& $j_1j_2^2 $& $ $\\

\hline

$\frak{sl}(n+1,\{\mathbb{R, C}\})$&$-1 $& $1 $& $1 $& $ \alpha_2,\alpha_p,\alpha_2+\dots \alpha_p$\\

$\{1,2,p\}$&$\sqrt[3]{1} $& $1$& $ \sqrt[3]{1}^2$& $ \alpha_2,\alpha_1+\dots+\alpha_p $\\

$(\alpha_1,\alpha_2)$&$ 1$& $j_2 $& $1 $& $ \alpha_1,\alpha_p $\\

$(2,0,-1)$&$j_1 $& $1 $& $ j_1^2$& $ \alpha_2 $\\

$\alpha_p, n>p> 3$&$-1$&$-1$& $1$& $\alpha_p,\alpha_1+\alpha_2, \alpha_1+\alpha_2+\dots+\alpha_p $\\

&$-1 $& $ j_2$& $1 $& $ \alpha_p $\\

&$ j_1$& $j_1^{-1} $& $j_1^2 $& $ \alpha_1+\alpha_2 $\\

&$ j_1$& $ j_1^{-2} $& $ j_1^2 $& $ \alpha_2+\dots+\alpha_p$\\

&$ j_1$& $ j_1^{-3}$& $j_1^2 $&$ \alpha_1+\dots +\alpha_p $\\

&$ j_1$& $ j_2$& $j_1^2 $& $ $\\

\hline

$\frak{sl}(n+1,\{\mathbb{R, C}\})$&$1 $& $-1 $& $-1 $& $ \alpha_1,\alpha_p+\dots+\alpha_n,\alpha_1+\dots+\alpha_n $\\

$\{1,p,n\}$&$-1 $& $-1$& $ 1$& $\alpha_n,\alpha_1+\dots+\alpha_p,\alpha_1+\dots+\alpha_n $\\

$(\alpha_1,\alpha_n)$&$ 1$& $j_p $& $j_p$& $ \alpha_1 $\\

$(1,-1,1)$&$j_1 $& $j_1 $& $ 1$& $ \alpha_n $\\

$\alpha_p, p> 2$&$\sqrt[3]{1}$& $\sqrt[3]{1}^2$& $\sqrt[3]{1}$&$\alpha_1+\dots+\alpha_p,\alpha_p +\dots+\alpha_n$\\

&$j_1$& $1$& $j_1^{-1} $& $ \alpha_p, \alpha_1+\dots+\alpha_n$\\

&$j_1$& $-1$& $-j_1^{-1}$&$\alpha_1+\dots+\alpha_n$\\

&$ j_1$& $j_1^{-1} $& $j_1^{-2} $& $ \alpha_1+\dots+\alpha_p $\\

&$ j_p^2$& $ j_p $& $ j_p^{-1} $& $ \alpha_p+\dots+\alpha_n$\\

&$ j_1$& $ j_1j_n$& $j_n$&$ $\\

\hline

$\frak{su}(2,2)$&$\sqrt[3]{1} $& $1$& $ \sqrt[3]{1}^2$& $ \alpha_2,\alpha_1+\dots+\alpha_3 $\\

$\{1,2,3\}$&$ 1$& $j_2 $& $1 $& $ \alpha_1,\alpha_3 $\\

$(\alpha_1,\alpha_{2})$&$ \sqrt[3]{1}$& $ j_2$& $\sqrt[3]{1}^2 $& $ $\\

$(2,0,-1)$&$ $& $ $& $ $& $ $\\

\hline

$\frak{su}(2,2)$& $1$& $-1$&$1$&$\alpha_1,\alpha_3$\\

$\{1,2,3\}$&$-1 $& $ 1$& $ -1$& $\alpha_2,\alpha_1+\dots+\alpha_3 $\\

$(\alpha_2,\alpha_{1})$&$-1 $& $ -1$& $ -1$& $\alpha_1+\alpha_2,\alpha_2+\alpha_3 $\\

$(1,2,-1)$&$e^r $& $ 1$& $e^r$& $ \alpha_2 $\\

&$e^r$& $-1 $& $e^r$& $ $\\

\hline

$\frak{su}(n,n)$&$e^{i\phi}$& $1$& $e^{-i\phi} $& $ \alpha_n, \alpha_1+\dots+\alpha_{2n-1}$\\

$\{1,n,2n-1\}$&$ j_1$& $ j_1\bar {j_1}$& $\bar {j_1}$&$ $\\

$(\alpha_1,\alpha_{2n-1})$&$ $& $ $& $ $& $ $\\

$(1,-1,1)$&$ $& $ $& $ $& $ $\\

$\alpha_n, n>2$&$ $& $ $& $ $& $ $\\

\hline

\end{tabular}

\begin{tabular}{|c||c|c|c|c|}

\hline

all info & $j_{i_1}$& $j_{i_2}$& $j_{i_3}$&$\frak{m}$ \\

\hline

\hline

$\frak{sp}(2n,\{\mathbb{R, C}\})$&$1 $& $-1 $& $1 $& $ \alpha_1,\alpha_n,\alpha_1+2\alpha_2+\dots $\\

$\{1,2,n\}$&$ -1$& $1 $& $-1 $& $ \alpha_2,\alpha_1+\dots+\alpha_n,\alpha_1+2\alpha_2+\dots $\\

$(\alpha_2,\alpha_1)$&$1 $& $\sqrt[3]{1}$& $ \sqrt[3]{1}^2$& $ \alpha_1,\alpha_2+\dots+\alpha_n,\alpha_1+\dots+\alpha_n$\\

$(1,2,-1)$&$1 $& $ \sqrt[4]{1}$& $ \sqrt[2]{1}$& $ \alpha_1,2\alpha_2+\dots,\alpha_1+2\alpha_2+\dots,2\alpha_1+\dots$\\

$\alpha_1, $&$\sqrt[3]{1} $& $1 $& $ \sqrt[3]{1}$& $ \alpha_2,\alpha_1+\dots+\alpha_n,2\alpha_1+\dots$\\

$\alpha_n, n>3$&$\sqrt[3]{1}^2 $& $\sqrt[3]{1} $& $\sqrt[3]{1} $& $\alpha_1+\alpha_2,2\alpha_2+\dots$\\

&$\sqrt[5]{1}^2 $& $\sqrt[5]{1} $& $\sqrt[5]{1}^4 $& $\alpha_2+\dots+\alpha_n,2\alpha_1+\dots$\\

&$\sqrt[5]{1} $& $\sqrt[5]{1} $& $\sqrt[5]{1}^3 $& $\alpha_1+\dots+\alpha_n,2\alpha_2+\dots$\\

&$\sqrt[2]{1} $& $\sqrt[8]{1} $& $\sqrt[4]{1}^3 $& $2\alpha_2+\dots,2\alpha_1+\dots $\\

&$-1$& $-1 $& $-1$& $\alpha_1+\alpha_2, \alpha_2+\dots+\alpha_n,\alpha_1+2\alpha_2+\dots $\\

&$ 1$& $j_2 $& $j_2^2 $& $\alpha_1$\\

&$j_2^{-2} $& $ j_2$& $1 $& $ \alpha_n,\alpha_1+2\alpha_2+\dots$\\

&$j_1 $& $1 $& $j_1 $& $ \alpha_2 $\\

&$ j_1$& $j_1^{-1} $& $j_1^{-1} $& $\alpha_1+\alpha_2 $\\

&$j_2^{-3} $& $j_2 $& $ j_2^{-1}$& $ \alpha_2+\dots+\alpha_n $\\

&$j_2^{-4} $& $j_2 $& $ j_2^{-2}$& $ 2\alpha_2+\dots$\\

&$j_n^{-6} $& $ j_n^2$& $ j_n$& $\alpha_1+\dots+\alpha_n $\\

&$ -j_2^{-2}$& $j_2 $& $ -1$& $\alpha_1+2\alpha_2+\dots$\\

&$\sqrt[3]{j_2^{-4}} $& $j_2 $& $\sqrt[3]{j_2^2} $& $2\alpha_1+\dots $\\

&$j_1 $& $j_2 $& $j_1j_2^2 $& $ $\\

\hline

$\frak{sp}(2n,\{\mathbb{R, C}\})$&$1 $& $j_2$& $\pm j_2$& $ \alpha_1$\\

$\{1,2,p\}$&$j_p^2$& $1 $& $j_p $& $\alpha_2$\\

$n>3,n>p$&$j_p^{-2} $& $ j_p^{2}$& $ j_p$& $\alpha_1+\alpha_2$\\

$(\alpha_2,\alpha_1)$&$j_p^{6} $& $ j_p^{-2}$& $ j_p$& $\alpha_2+\dots+2\alpha_p+\dots$\\

$(1,2,-2)$&$1$& $\sqrt[3]{1}^2$& $\sqrt[6]{1}$& $\alpha_1,\alpha_2+\dots+2\alpha_p+\dots,\alpha_1+\alpha_2+\dots+2\alpha_p+\dots$\\

$\alpha_1,\alpha_p $&$j_p^{-6} $& $ j_p^4$& $ j_p$& $\alpha_1+\alpha_2+\dots+2\alpha_p+\dots$\\

&$\sqrt[3]{1}^2 $& $\sqrt[3]{1}$& $\sqrt[3]{1}^{2}$& $\alpha_1+\alpha_2,\alpha_2+\dots+\alpha_p,2\alpha_2+\dots$\\

&$\sqrt[5]{1}$& $\sqrt[5]{1}$& $\sqrt[5]{1}^4$& $\alpha_2+\dots+\alpha_p,\alpha_1+\alpha_2+\dots+2\alpha_p+\dots,2\alpha_2+\dots$\\

&$j_2^{-4} $& $ j_2$& $ j_2^{-1}$& $\alpha_2+\dots+\alpha_p,2\alpha_2+\dots$\\

&$\sqrt[3]{1}^2 $& $\sqrt[3]{1}$& $-\sqrt[3]{1}^{2}$& $\alpha_1+\alpha_2,2\alpha_2+\dots$\\

&$\sqrt[5]{1}$& $\sqrt[5]{1}$& $-\sqrt[5]{1}^4$& $\alpha_1+\alpha_2+\dots+2\alpha_p+\dots,2\alpha_2+\dots$\\

&$j_2^{-4} $& $ j_2$& $-j_2^{-1}$& $2\alpha_2+\dots$\\

&$-1$& $1 $& $\sqrt[4]{1}$& $\alpha_2, \alpha_1+\alpha_2+\dots+2\alpha_p+\dots,\alpha_1+2\alpha_2+\dots $\\

&$-1$& $-1 $& $\sqrt[4]{1}$& $\alpha_1+\alpha_2, \alpha_2+\dots+2\alpha_p+\dots,\alpha_1+2\alpha_2+\dots $\\

&$-j_2^{-2}$& $j_2 $& $\sqrt[4]{1}$& $\alpha_1+2\alpha_2+\dots $\\

&$1 $& $-1 $& $-1 $& $\alpha_1,\alpha_2+\dots +\alpha_p, \alpha_1+\dots +\alpha_p,2\alpha_p+\dots,$\\

&$ $& $ $& $ $& $2\alpha_p+\dots,2\alpha_2+\dots,\alpha_1+2\alpha_2+\dots,2\alpha_1+\dots $\\

&$\sqrt[3]{1}^2$& $1 $& $\sqrt[3]{1} $& $\alpha_2,\alpha_1+\dots +\alpha_p,2\alpha_1+\dots $\\

&$1$& $\sqrt[4]{1}^3 $& $\sqrt[4]{1}$& $\alpha_1,\alpha_2+\dots +\alpha_p,\alpha_1+\dots +\alpha_p,$\\

&$ $& $ $& $ $& $2\alpha_2+\dots, \alpha_1+2\alpha_2+\dots, 2\alpha_1+\dots $\\

&$\sqrt[5]{1}$& $\sqrt[5]{1}^3$& $\sqrt[5]{1} $& $\alpha_1+\dots +\alpha_p,\alpha_2+\dots+2\alpha_p+\dots,2\alpha_1+\dots $\\

&$-1$& $\sqrt[8]{1}^3 $& $\sqrt[8]{1}$& $\alpha_1+\dots +\alpha_p,2\alpha_2+\dots, 2\alpha_1+\dots $\\

&$j_p^{-4}$& $j_p^3 $& $j_p $& $\alpha_1+\dots +\alpha_p,2\alpha_1+\dots $\\

&$1 $& $-1$& $ 1$& $\alpha_1,\alpha_p, 2\alpha_p+\dots, 2\alpha_2+\dots,\alpha_1+2\alpha_2+\dots,2\alpha_1+\dots$\\

&$ 1$& $1 $& $-1 $& $ \alpha_1,\alpha_2,\alpha_1+\alpha_2, 2\alpha_p+\dots,\alpha_2+\dots+2\alpha_p+\dots, 2\alpha_2+\dots,$\\

&$ $& $ $& $ $& $\alpha_1+\alpha_2+\dots+2\alpha_p+\dots, \alpha_1+2\alpha_2+\dots,2\alpha_1+\dots$\\

&$1$& $-\sqrt[4]{1}^3 $& $\sqrt[4]{1}$& $\alpha_1,\alpha_2+\dots +\alpha_p,2\alpha_2+\dots, \alpha_1+2\alpha_2+\dots, 2\alpha_1+\dots $\\

&$-1$& $-\sqrt[8]{1}^3 $& $\sqrt[8]{1}$& $\alpha_2+\dots+\alpha_p,2\alpha_2+\dots, 2\alpha_1+\dots $\\

&$\sqrt[3]{1}$& $1 $& $\sqrt[6]{1} $& $\alpha_2,2\alpha_1+\dots $\\

&$\sqrt[10]{1}^6$& $-\sqrt[10]{1}^3$& $\sqrt[10]{1} $& $\alpha_2+\dots+2\alpha_p+\dots,2\alpha_1+\dots $\\

&$j_p^{-4}$& $-j_p^3 $& $j_p $& $2\alpha_1+\dots $\\

&$j_2^{-2} $& $j_2 $& $1 $& $\alpha_p,2\alpha_p+\dots,\alpha_1+2\alpha_2+\dots, $\\

&$j_2^{-2} $& $j_2 $& $-1 $& $2\alpha_p+\dots,\alpha_1+2\alpha_2+\dots, $\\

&$j_2^{-2}j_p^2 $& $j_2 $& $j_p $& $ $\\

\hline

\end{tabular}

\noindent
\resizebox{\textwidth}{!}{
\begin{tabular}{|c||c|c|c|c|}

\hline

all info & $j_{i_1}$& $j_{i_2}$& $j_{i_3}$&$\frak{m}$ \\

\hline

\hline

$\frak{so}(3,4)$, $\frak{so}(7,\mathbb{C})$&$1 $& $1 $& $\sqrt[3]{1}$& $ \alpha_1,\alpha_2, \alpha_1+\alpha_2 $\\

$\{1,2,3\}$&$ \sqrt[4]{1}^3 $& $1$& $ \sqrt[4]{1}$& $\alpha_2,\alpha_1+\alpha_2+\alpha_3$\\

$(\alpha_3,\alpha_2)$&$ \sqrt[5]{1}^3$& $1 $& $ \sqrt[5]{1} $& $\alpha_2, \alpha_1+\alpha_2+2\alpha_3,\alpha_1+2\alpha_2+2\alpha_3$\\

$(-1,0,3)$&$\sqrt[4]{1}^3 $& $\sqrt[4]{1}^3 $& $\sqrt[4]{1} $& $\alpha_2+\alpha_3, \alpha_1+\alpha_2+2\alpha_3$\\

&$-1 $& $-1 $& $-1$& $\alpha_1+\alpha_2,\alpha_2+\alpha_3,\alpha_1+\alpha_2+2\alpha_3 $\\

&$1 $& $-1$& $ 1$& $ \alpha_1,\alpha_3, \alpha_1+2\alpha_2+2\alpha_3$\\

&$-1 $& $1 $& $-1 $& $\alpha_2,\alpha_1+\alpha_2+\alpha_3,\alpha_2+2\alpha_3$\\

&$1 $& $j_2$& $ 1$& $\alpha_1,\alpha_3$\\

&$1 $& $j_2$& $\sqrt[3]{1}$& $\alpha_1$\\

&$j_3^3 $& $1 $& $j_3 $& $\alpha_2 $\\

&$ 1$& $\sqrt[3]{1}^2$& $\sqrt[3]{1}$& $\alpha_1,\alpha_2+\alpha_3, \alpha_1+\alpha_2+\alpha_3,\alpha_1+2\alpha_2+2\alpha_3$\\

&$1 $& $\sqrt[3]{1}$& $\sqrt[3]{1}$& $\alpha_1,\alpha_2+2\alpha_3, \alpha_1+\alpha_2+2\alpha_3$\\

&$1 $& $\sqrt[6]{1}$& $\sqrt[3]{1}$& $\alpha_1,\alpha_1+2\alpha_2+2\alpha_3$\\

&$ j_1$& $j_1^{-1} $& $\sqrt[3]{j_1} $& $\alpha_1+\alpha_2$\\

&$j_2^{-3} $& $j_2 $& $ j_2^{-1}$& $\alpha_2+\alpha_3$\\

&$j_3^{3} $& $j_3^{-4} $& $ j_3$& $\alpha_1+\alpha_2+\alpha_3$\\

&$j_3^{3} $& $ j_3^{-2}$& $ j_3$& $\alpha_2+2\alpha_3$\\

&$ j_3^{3}$& $j_3^{-5} $& $ j_3$& $\alpha_1+\alpha_2+2\alpha_3$\\

&$j_3^{3} $& $\pm j_3^{-5/2} $& $j_3 $& $\alpha_1+2\alpha_2+2\alpha_3$\\

&$j_3^3 $& $j_2 $& $j_3 $& $ $\\

\hline

$\frak{so}(3,5)$& $-1$ &$1 $& $ 1$& $\alpha_3,\alpha_4,2\alpha_2+\alpha_3+\alpha_4$\\

$\{2,3,4\}$&$ j_2$& $1 $& $ 1$& $ \alpha_3,\alpha_4$\\

$(\alpha_4,\alpha_2)$& $1$&$\sqrt[3]{1}$& $\sqrt[3]{1}^2 $& $\alpha_2, \alpha_2+\alpha_3+\alpha_4,2\alpha_2+\alpha_3+\alpha_4$\\

$(0,-1,2)$& $j_2 $&$\sqrt[3]{1} $& $\sqrt[3]{1}^2 $& $ $\\

\hline

$\frak{so}(n,n)$, $\frak{so}(2n,\mathbb{C})$&$-1 $& $-1 $& $1 $& $\alpha_n,\alpha_1+\alpha_2,\alpha_1+\dots+\alpha_n$\\

$\{1,2,n\}$&$1 $& $-1$& $ 1$& $\alpha_1,\alpha_n,\alpha_1+2\alpha_2+\dots$\\

$(\alpha_1,\alpha_2)$&$ -1$& $1 $& $1 $& $ \alpha_2,\alpha_n,\alpha_2+\dots+\alpha_n$\\

$(2,0,-1)$&$-1 $& $\sqrt[4]{1}$& $1$& $ \alpha_n, \alpha_1+2\alpha_2+\dots$\\

$\alpha_2,$&$-1 $& $j_2$& $1$& $ \alpha_n$\\

$\alpha_n, n>4$&$1 $& $ j_2$& $ 1$& $ \alpha_1,\alpha_n$\\

&$\sqrt[3]{1}$& $1$& $\sqrt[3]{1}^2 $& $\alpha_2, \alpha_1+\dots+\alpha_n,\alpha_1+2\alpha_2+\dots $\\

&$ j_1$& $j_1^{-1} $& $j_1^{2} $& $\alpha_1+\alpha_2$\\

&$j_1$& $j_1^{-2}$& $j_1^{2}$& $\alpha_2+\dots+\alpha_n $\\

&$j_1$& $j_1^{-3} $& $j_1^2 $& $\alpha_1+\dots+\alpha_n $\\

&$ j_1$& $\pm\sqrt{ j_1}^{-3} $& $j_1^{2} $& $\alpha_1+2\alpha_2+\dots $\\

&$j_1 $& $1 $& $j_1^2 $& $\alpha_2 $\\

&$j_1 $& $j_2 $& $j_1^2 $& $ $\\

\hline
\end{tabular}
}

\noindent
\resizebox{\textwidth}{!}{
\begin{tabular}{|c||c|c|c|c|c|}

\hline

all info & $j_{i_1}$& $j_{i_2}$& $j_{i_3}$& $j_{i_4}$&$\frak{m}$ \\

\hline

\hline

$\frak{sl}(n+1,\{\mathbb{R, C}\})$ &$1 $&$j_2 $&$j_p $&$j_2^2j_p^{-1}$&$\alpha_1 $\\

$\{1,2,p,q\}$&$1 $&$1 $&$ j_p$&$j_p^{-1} $&$\alpha_1,\alpha_2,\alpha_1+\alpha_2, \alpha_p+\dots +\alpha_q, \alpha_2+\dots +\alpha_q,\alpha_1+\dots+\alpha_q $\\

$(\alpha_2,\alpha_1)$&$1 $&$-1 $&$j_p $&$j_p^{-1} $&$\alpha_1, \alpha_p+\dots +\alpha_q $\\

$(1,2,-1,-1)$&$j_2^{-3} $&$j_2 $&$1 $&$j_2^{-1} $&$\alpha_p, \alpha_2+\dots +\alpha_q $\\

$\alpha_1,$&$1 $&$j_2 $&$j_2^2 $&$1$&$\alpha_1,\alpha_q $\\

$\alpha_p,\alpha_q$&$j_2^{-1} $&$j_2 $&$1 $&$j_2 $&$\alpha_p,\alpha_1+\alpha_2, \alpha_1+\dots \alpha_p $\\

$p,q\neq n,$&$1 $&$j_2 $&$j_2^{-1} $&$ j_2^3$&$\alpha_1,\alpha_2+\dots +\alpha_p, \alpha_1+\dots +\alpha_p $\\

$n>4,p>3$&$j_1 $&$ j_2$&$1 $&$j_1j_2^2 $&$\alpha_p $\\

&$j_1 $&$ 1$&$ 1$&$j_1 $&$\alpha_2,\alpha_p,\alpha_2+\dots +\alpha_p $\\

&$1 $&$j_2 $&$1 $&$ j_2^2$&$\alpha_1,\alpha_p $\\

&$j_2^{-2} $&$j_2 $&$1 $&$1 $&$\alpha_p,\alpha_q,\alpha_p+\dots +\alpha_q $\\

&$-1 $&$1 $&$1 $&$-1 $&$\alpha_2,\alpha_p, \alpha_2+\dots +\alpha_p,\alpha_1+\dots +\alpha_q $\\

&$-1 $&$1 $&$-1 $&$1 $&$\alpha_2,\alpha_q, \alpha_1+\dots +\alpha_p,\alpha_1+\dots +\alpha_q $\\

&$1 $&$-1 $&$1 $&$1 $&$\alpha_1,\alpha_p,\alpha_q, \alpha_p+\dots +\alpha_q $\\

&$ 1$&$ -1$&$-1 $&$-1 $&$\alpha_1,\alpha_2+\dots +\alpha_p,\alpha_p+\dots +\alpha_q,\alpha_1+\dots +\alpha_p $\\

&$ 1$&$\sqrt[3]{1} $&$1$&$\sqrt[3]{1}^2 $&$\alpha_1,\alpha_p,\alpha_2+\dots +\alpha_q,\alpha_1+\dots +\alpha_q $\\

&$ 1$&$\sqrt[3]{1} $&$j_p $&$\sqrt[3]{1}^2j_p^{-1} $&$\alpha_1,\alpha_2+\dots +\alpha_q,\alpha_1+\dots +\alpha_q $\\

&$ 1$&$\sqrt[3]{1} $&$\sqrt[3]{1}^2 $&$1 $&$\alpha_1,\alpha_q, \alpha_2+\dots +\alpha_p,\alpha_1+\dots +\alpha_p, \alpha_2+\dots +\alpha_q,\alpha_1+\dots +\alpha_q $\\

&$-1 $&$ -1$&$ 1$&$-1 $&$\alpha_p,\alpha_1+\alpha_2,\alpha_1+\dots +\alpha_p, \alpha_2+\dots +\alpha_q $\\

&$\pm \sqrt{j_2}^{-3} $&$j_2 $&$1 $&$\pm \sqrt{j_2}$&$\alpha_p,\alpha_1+\dots +\alpha_q $\\

&$j_1 $&$j_2 $&$j_1j_2^2 $&$1 $&$\alpha_q $\\

&$j_2^{-1} $&$j_2 $&$j_2 $&$1 $&$\alpha_q,\alpha_1+\alpha_2 $\\

&$j_1 $&$1 $&$ j_1$&$1 $&$\alpha_2, \alpha_q $\\

&$-1 $&$-1 $&$-1 $&$1 $&$\alpha_q,\alpha_1+\alpha_2,\alpha_2+\dots +\alpha_p,\alpha_2+\dots +\alpha_q $\\

&$j_2^{-3} $&$j_2 $&$j_2^{-1} $&$1 $&$ \alpha_q,\alpha_2+\dots +\alpha_p,\alpha_2+\dots +\alpha_q$\\

&$\pm \sqrt{j_2}^{-3} $&$j_2 $&$\pm \sqrt{j_2}$&$1$&$\alpha_q,\alpha_1+\dots +\alpha_p,\alpha_1+\dots +\alpha_q $\\

&$j_1 $&$1 $&$j_p $&$j_1j_p^{-1} $&$ \alpha_2$\\

&$ j_1$&$1 $&$j_1^{-1} $&$j_1^{2} $&$\alpha_2,\alpha_1+\dots +\alpha_p $\\

&$ -1$&$1 $&$ j_p$&$-j_p^{-1} $&$ \alpha_2,\alpha_1+\dots +\alpha_q $\\

&$j_1 $&$j_1^{-1} $&$j_p $&$(j_1j_p)^{-1}$&$\alpha_1+ \alpha_2 $\\

&$-1 $&$ -1$&$ j_p$&$-j_p^{-1} $&$\alpha_1+ \alpha_2,\alpha_2+\dots +\alpha_q $\\

&$ j_1$&$j_2 $&$ j_2^{-1}$&$ j_1j_2^3$&$\alpha_2+\dots +\alpha_p $\\

&$ j_1$&$j_1^{-1} $&$j_1 $&$j_1^{-2} $&$\alpha_1+ \alpha_2,\alpha_2+\dots +\alpha_p $\\

&$j_p^2 $&$j_p^{-1} $&$j_p $&$j_p^{-1} $&$ \alpha_2+\dots +\alpha_p,\alpha_p+\dots +\alpha_q$\\

&$\pm \sqrt{j_p}^{3} $&$j_p^{-1} $&$j_p $&$\pm \sqrt{j_p}^{-3} $&$\alpha_2+\dots +\alpha_p,\alpha_1+\dots +\alpha_q $\\

&$j_2^{-2} $&$ j_2$&$j_p $&$j_p^{-1} $&$\alpha_p+\dots +\alpha_q $\\

&$j_2^{-2} $&$ j_2$&$ j_2$&$j_2^{-1} $&$\alpha_p+\dots +\alpha_q,\alpha_1+\dots +\alpha_p $\\

&$(j_2j_p)^{-1} $&$j_2 $&$ j_p$&$j_2/j_p^2 $&$ \alpha_1+\dots +\alpha_p$\\

&$j_2^{-3} $&$j_2 $&$j_2^2 $&$j_2^{-3} $&$\alpha_1+\dots +\alpha_p,\alpha_2+\dots +\alpha_q $\\

&$j_2^{-3} $&$j_2 $&$ j_p$&$(j_1j_p)^{-1}$&$\alpha_2+\dots +\alpha_q $\\

&$\pm \sqrt{j_2}^{-3} $&$j_2 $&$j_p $&$\pm \sqrt{j_2}j_p^{-1} $&$\alpha_1+\dots +\alpha_q $\\

&$j_1 $&$j_2 $&$j_p $&$j_1j_2^2j_p^{-1} $&$ $\\

\hline

$\frak{su}(q,n-q+1)$ &$ 1$&$\sqrt[3]{1} $&$\sqrt[3]{1}^2 $&$1 $&$\alpha_1,\alpha_q, \alpha_2+\dots +\alpha_p,\alpha_1+\dots +\alpha_p, \alpha_2+\dots +\alpha_q,\alpha_1+\dots +\alpha_q $\\

$\{1,2,n-1,n\}$&$1 $&$-1 $&$-1 $&$1 $&$\alpha_1, \alpha_p+\dots +\alpha_q,\alpha_q,\alpha_1+\alpha_2 $\\

$(\alpha_2,\alpha_1)$&$e^r $&$ 1$&$1 $&$e^r$&$\alpha_2,\alpha_p,\alpha_2+\dots +\alpha_p $\\

$(1,2,-1,-1)$&$-1 $&$1 $&$1 $&$-1 $&$\alpha_p,\alpha_1+\dots +\alpha_p,\alpha_1+\dots +\alpha_q $\\

&$ e^{r-\frac32i\phi}$&$e^{i\phi}$&$ e^{-i\phi}$&$ e^{r+\frac32i\phi}$&$\alpha_2+\dots +\alpha_p $\\

&$\sqrt[2]{1} $&$\sqrt[3]{1}^{2} $&$\sqrt[3]{1}$&$\sqrt[2]{1} $&$\alpha_2+\dots +\alpha_p,\alpha_1+\dots +\alpha_q $\\

\hline

\end{tabular}
}


\begin{thebibliography}{99}

\bibitem{parabook} A. \v Cap, J. Slov\'ak, 
``Parabolic Geometries I:
Background and General Theory'', Amer. Math. Soc., 2009

\bibitem{Cap-correspondence} A. \v Cap, \emph{Correspondence spaces and twistor spaces for parabolic geometries}, J. Reine Angew. Math. 582, (2005), 143-172

\bibitem{Cap-holonomy} A. \v Cap,  M. Hammerl,  A.R. Gover, \emph{Holonomy reductions of Cartan geometries and curved orbit decompositions}, Duke Math. J. 163, no. 5, (2014), 1035-1070
  
\bibitem{G2}J.  Gregorovi\v c , \emph{General construction of symmetric parabolic geometries}, Differential Geometry and its Applications 30, (2012), 450-476 

\bibitem{disertace} J. Gregorovi\v c, ``Geometric structures invariant to symmetries'', Masaryk University, 2012

\bibitem{GZ} J. Gregorovi\v c, L. Zalabov\' a, \emph{Symmetric parabolic contact geometries and symmetric spaces}, Transformation Groups, Volume 18, (2013), Issue 3 (September), 711-737 

\bibitem{GZ2} J. Gregorovi\v c, L. Zalabov\' a, \emph{On automorphisms with natural tangent action on homogeneous parabolic geometries},  	
Journal of Lie Theory 25 (2015), No. 3,  677--715

\bibitem{Kn96}   A.W. Knapp, ``Lie Groups Beyond an Introduction'', Birkh\"auser, 1996
 
\bibitem{KT} B. Kruglikov, D. The, \emph{The gap phenomena in parabolic geometries},  arXiv:1303.1307v4

\bibitem{kowalski} O. Kowalski, Generalized Symmetric spaces, Lecture Notes in Mathematics, Vol. 805, Springer-Verlag, 1980

\bibitem{Lie-alg-3}  A.L. Onishchik,  E.B. Vinberg, (Eds.)
Lie Groups and Lie Algebras III: Structure of Lie Groups and Lie Algebras,
Series: Encyclopaedia of Mathematical Sciences, Vol. 41, Springer, 1991

\bibitem{LZ1} L. Zalabov\' a, \emph{Symmetries of Parabolic Geometries}, Differential Geometry and its Applications {\bf 27}, (2009), 605-622

\bibitem{LZ2} L. Zalabov\' a, \emph{Symmetries of Parabolic Contact Structures}, Journal of Geometry and Physics {\bf 60}, (2010), 1698-1709
\end{thebibliography}
\end{document}